\documentclass[11pt,a4]{amsart}

\usepackage{comment}
\usepackage{amsmath,amsthm,amssymb,graphicx,fancyhdr,mathtools}
\usepackage{color}
\usepackage[top=1in, bottom=1in, left=1in, right=1in]{geometry}
\usepackage{enumerate}
\usepackage[titletoc,title]{appendix}
\usepackage{bbm}
\usepackage{bm}
\usepackage{amsbsy}
\usepackage{mathrsfs}
\usepackage{dsfont}

\DeclareMathAlphabet{\mathpzc}{OT1}{pzc}{m}{it}

\usepackage[all,cmtip]{xy}

\newtheorem{theorem}{Theorem}[section]
\newtheorem{definition}{Definition}[section]
\newtheorem{lemma}{Lemma}[section]
\newtheorem{corollary}{Corollary}[section]
\newtheorem{assumption}{Assumption}
\newtheorem{remark}{Remark}[section]
\newtheorem{example}{Example}[section]
\newtheorem{proposition}{Proposition}[section]
   \newtheoremstyle{example}{\topsep}{\topsep}%
     {}
     {}
     {\bfseries}
     {}
     {\newline}
     {\thmname{#1}\thmnumber{ #2}\thmnote{ #3}}

   \theoremstyle{example}
\newcommand{\E}{\mathbb{E}}

\newcommand{\F}{\mathcal{F}}

\newcommand{\N}{\mathbb{N}}
\newcommand{\Prob}{\mathbb{P}}
\newcommand{\Poly}{\mathcal{P}}

\newcommand{\calD}{\mathcal{D}}
\newcommand{\calL}{\mathcal{L}}
\newcommand{\calJ}{{\mathcal{J}}}
\newcommand{\R}{\mathbb{R}}
\newcommand{\m}{\mathfrak{m}}

\newcommand{\p}{\mathfrak{p}}
\newcommand{\Nu}{\mathcal{V}}
\newcommand{\Tau}{\mathcal{T}}

\newcommand{\II}{\mathds{I}}
\newcommand{\bL}{\mathbf{L}}

\newcommand{\B}{\mathbf{B}}
\newcommand{\lam}{{\bm{\lambda}_1}}

\newcommand{\opphi}{\partial_t^{\star \phi}}

\newcommand{\opma}{\partial_t^{\star m_{\alpha}}}
\newcommand{\opm}{\partial_t^{\star m}}

\newcommand{\op}{\partial_t^{\star m,b}}
\newcommand{\opPhi}{\partial_t^{\star \Phi}}

\newcommand{\opbPhi}{\partial_t^{\star \pmb{\phi}}}

\newcommand{\Ep}{\mathcal{E}_{\phi_{\alpha}}}

\newcommand{\Ehat}{\widehat{\mathcal{E}}_{\phi_{\alpha}}}
\newcommand{\bPhi}{\pmb{\phi}}
\newcommand{\e}{\mathbf{e}}
\newcommand{\ha}{{\mathcal{H}}_{\gamma}}

\newcommand{\hha}{\widehat{\mathcal{H}}_{\gamma}}
\newcommand{\hatJ}{\widehat{\mathcal{J}}_{\gamma}}
\newcommand{\pa}{{\pi}_{\gamma}}

\newcommand{\DL}{\mathcal{D}_{\bL}}
\newcommand{\BD}{\B_{\partial^\star }}

\numberwithin{equation}{section}

\author{P. Patie}\thanks{The authors are indebted to Mark Meerschaert for providing them many interesting references on the spectral approach in the context of the fractional Cauchy problem and also for his invaluable encouragements. They are also grateful to the referees for careful reading, constructive comments and providing several interesting references on different aspects of the paper.}
\address{School of Operations Research and Information Engineering, Cornell University, Ithaca, NY 14853.}
\email{	pp396@orie.cornell.edu}

\author{A. Srapionyan}
\address{Center for Applied Mathematics, Cornell University, Ithaca, NY 14853.}
\email{	as3348@cornell.edu}

\usepackage{hyperref}

\title{Self-similar Cauchy problems and generalized Mittag-Leffler functions}

\begin{document}
\begin{abstract}
By observing that the  fractional Caputo derivative of order $\alpha \in (0,1)$ can be expressed in terms of a multiplicative convolution operator, we introduce and study a class of such operators which also have the same self-similarity property as the Caputo derivative. We proceed by identifying a subclass which is in bijection with the set of Bernstein functions and we provide several representations of their eigenfunctions, expressed in terms of the corresponding Bernstein function,  that generalize the Mittag-Leffler function.  Each eigenfunction turns out to be the Laplace transform of the right-inverse of a non-decreasing self-similar Markov process associated via the so-called Lamperti mapping to this Bernstein function. Resorting to spectral theoretical arguments, we investigate the generalized Cauchy problems, defined with these self-similar multiplicative convolution operators. In particular, we provide both a stochastic representation, expressed in terms of these inverse processes, and an explicit representation, given in terms of the generalized Mittag-Leffler functions, of the solution of these self-similar  Cauchy problems. This work could be seen as an-in depth analysis of a specific  class, the one with the  self-similarity property, of the general  inverse of increasing Markov processes introduced and studied  in \cite{Kol1}.
\end{abstract}

\maketitle




\section{Introduction} \label{sec:intro}
The fractional Caputo derivative of order $\alpha \in (0,1)$ which is usually defined in terms of the additive convolution operator $*$  and the function $h_{\alpha}(y)= \frac{y^{-\alpha}}{\Gamma(1-\alpha)}, y>0,$ as follows
\begin{equation}\label{Caputo}
\frac{~^Cd^{\alpha}}{dt^{\alpha}}f(t) = f'*h_{\alpha}(t)=\frac{1}{\Gamma(1-\alpha)}\int_0^t \frac{f'(y)}{(t-y)^{\alpha}}dy
\end{equation}
plays a central and growing role in various contexts, see e.g.~the monographs \cite{Maina, meerschaert_sikorskii, Book-Fract, fractional_applications}. In particular, in analysis, it  appears in the fractional Cauchy problem, where one replaces the derivative of order $1$ by the fractional one, i.e.~$\frac{~^C d^{\alpha}}{dt^{\alpha}}f = \mathbf L f$, with  $\mathbf L$ the infinitesimal generator of a strong Markov process $X$, see \cite{Zasla} for the introduction of this problem in relation to some Hamiltonian chaotic dynamics of particles given in terms of stable processes.

Bauemer and Meerschaert in \cite{baeumer_meerschaert} showed the intriguing fact that the solution of this problem admits a stochastic representation which is given in terms of a non-Markovian process  defined as the Markov process $X$ time-changed by the inverse of an $\alpha$-stable subordinator. This offers another fascinating connection between stochastic and functional analysis. Observing that the mapping $h_{\alpha}$ is the tail of the L\'evy measure of this stable subordinator, it is then natural to generalize the fractional operator as an additive convolution operator by replacing the function $h_{\alpha}$ with the tail of the L\'evy measure of any subordinator.  It turns out that this interesting program has been developed recently by Toaldo \cite{toaldo} and the corresponding generalized fractional Cauchy problem has, when this tail has infinite mass, a similar stochastic representation where the time-changed process is the inverse of the subordinator, see \cite{toaldo,chen2017time}, and, see the recent paper \cite{MB} where the independency assumption is relaxed.

Another important feature of the fractional Caputo derivative is its self-similarity property
\begin{equation} \label{eq:def_scaling}
\frac{~^C d^{\alpha}}{dt^{\alpha}} d_c f(t) = c^{\alpha} \frac{~^C d^{\alpha}}{dt^{\alpha}}f(ct), \quad c,t>0,
\end{equation}
where $d_cf(t)=f(ct)$ is the dilation operator. It is not difficult to convince yourself that this property follows from the homogeneity property of the function $h_{\alpha}$ which itself is inherited from  the scaling property of the $\alpha$-stable subordinator, and thus it does not hold for any $*$-convolution operators associated to any other subordinators. This property is appealing from a modelling viewpoint as it has been observed in many physical and economics phenomena \cite{Embrechts-Maejima}  and is also central  in (non-trivial) limit theorems for any properly normalized stochastic processes, see \cite{lamperti1962semi}. Two questions then arise naturally:
\begin{enumerate}
  \item \label{Q1} Can one define a class of linear operators enjoying the same self-similarity property \eqref{eq:def_scaling} as the fractional derivative?
  \item \label{Q2} If yes, can one find a stochastic representation for the solution of the corresponding self-similar Cauchy problem?
\end{enumerate}
The aim of this paper is to provide a positive and detailed answer to each of these questions.
For \eqref{Q1}, we observe that the fractional derivative~\eqref{Caputo} admits also the  representation as a multiplicative convolution operator
\begin{equation}\label{eq:Cap}
f'*h_{\alpha}(t)=\frac{t^{-\alpha}}{\Gamma(1-\alpha)}\int_0^t f'(y)\left(1-\frac{y}{t}\right)^{-\alpha}dy =t^{-\alpha} f' \star g_{\alpha}(t)
\end{equation}
where $g_{\alpha}(r)=\frac{(1-r)^{-\alpha}}{\Gamma(1-\alpha)} \mathbb{I}_{\{0<r<1\}}$ and $\star$ stands for the multiplicative convolution operator defined for two functions $f$ and $g$, whose domain is  a subset of $\R^+$, by \[ f \star g(t) = \int_0^{\infty} f(r) g\left( \frac{r}{t} \right)dr.\]
Note that it differs from the Mellin convolution operator which is defined with respect to the measure $dr/r$.
 It is not difficult to show that the self-similarity property \eqref{eq:def_scaling} holds for any $\star$-convolution operator of the form  \eqref{eq:Cap} by replacing $g_\alpha$ by any measurable function $m$ on $(0,1)$. The answer to the question \eqref{Q2} is more subtle. Indeed, we first realize that the mapping $y\mapsto g_{\alpha}(e^{-y})$ is the tail of the L\'evy measure on $\R_+$ of  a subordinator. We manage to identify this subordinator as the L\'evy process which is associated, via the Lamperti transform defined in \eqref{eq:def_lamp} below, to the stable subordinator seen as an increasing positive self-similar Markov process. We then show, by a spectral theoretical approach, that the functional of a Markov process $X$  time-changed by the inverse of any increasing positive self-similar Markov process $\chi$ is the solution to a self-similar Cauchy problem, where  the multiplicative convolution is defined in terms of simple transform of the tail of the L\'evy measure of the subordinator associated, via the Lamperti mapping, to $\chi$. We mention that such a time-change has already been used  in Loeffen and al.~\cite{semiMarkov_LPS} to provide detailed distributional properties of the extinction time of some real-valued non-Markovian self-similar  processes. We also point out that such a time-change falls in the framework developed by Hernandez-Hernandez et al.~\cite{Kol}  and  Kolokoltsov \cite{Kol1,Kol0,Kol2} in a series of papers.


Another interesting aspect of the multiplicative convolution approach is that it leads to some explicit representations of quantities of interest. For instance, we shall show that the Laplace transform of the inverse process, in this context, is expressed in terms of functions whose series representation is a generalization of the Mittag-Leffler function. As this latter for the Caputo fractional operator, these functions are  also eigenfunctions of the multiplicative convolution operators.

We now recall that due to the non-locality of fractional derivatives and integrals, fractional models provide a powerful tool for a description of memory and hereditary properties of different substances, see e.g.~Liu et al.~\cite{liu2015numerical} and Podlubny \cite{podlubny1999fractional}. Equations of fractional order appear in a lot of physical phenomena, see e.g.~Meerschaert and Sikorskii~\cite{meerschaert_sikorskii}, and in particular for modeling anomalous diffusions, see e.g.~Benson et al.~\cite{benson2001fractional}, D'Ovidio~\cite{d2012sturm} and Mainardi's monograph \cite{Maina}.
Fractional calculus, which defines and studies derivatives and integrals of fractional order, has been applied in various areas of engineering, science, finance, applied mathematics, and bio engineering.
The fractional Cauchy problems replace the integer time derivative by its fractional counterpart, i.e.~ $\frac{~^C d^{\alpha}}{dt^{\alpha}}f = \mathbf L f$. The connection between fractional Cauchy problems and the inverse of a stable subordinator was explored by many authors, see e.g.~ Baeumer and Meerschaert~\cite{baeumer_meerschaert, BaMe}, Meerschaert et~al.~\cite{meerschaert2009fractional}, Saichev and Zaslavsky~\cite{saichev1997fractional}, Zaslavsky~\cite{zaslavsky2002chaos} and Capitanelli and D'Ovidio \cite{Cap}, among others. 

The rest of this paper is organized as follows. In Section~\ref{sec:self-sim}, we study some of the substantial properties of the inverse of an increasing self-similar Markov process. In Section~\ref{sec:frac_op}, we introduce a self-similar multiplicative convolution generalization of the fractional Caputo derivative. In Section~\ref{sec:frac_Cauchy}, we study the corresponding self-similar Cauchy problem and provide the stochastic representation of its solution. Finally, to illustrate some examples, Section~\ref{sec:Examples} considers families of some self-adjoint, as well as, some non-local and non-self-adjoint Markov semigroups.

\section{Inverse of increasing self-similar Markov processes}\label{sec:self-sim}
Let $\chi=(\chi_t)_{t\geq 0}$ be  a non-decreasing self-similar Markov process of index $\alpha \in (0,1)$ issued from $0$ and denote by $\zeta=(\zeta_t)_{t\geq 0}$ its right-inverse, that is, for any $t\geq 0$,
\begin{equation}\label{eq:def_inv}
 \zeta_t=\inf \{s>0;\: \chi_s>t \}.
\end{equation}
 Denoting the law of the process by $\Prob_x$ when starting from $x>0$, we say that  a stochastic process $\chi$ is self-similar of index $\alpha$ (or $\alpha$-self-similar) if the following identity
\begin{equation}\label{eq:defselfsim}
  (c\chi_{c^{-\alpha}t},\Prob_x)_{t\geq 0} \stackrel{d}{=} (\chi_t,\Prob_{cx})_{t\geq 0}
\end{equation}
holds in the sense of finite-dimensional distribution for any $c>0$.
Now, we recall that Lamperti~\cite{lamperti1972semi} identifies a one-to-one mapping between the class of positive self-similar Markov processes and the one of L\'evy processes. In particular, one has, under $\Prob_x, x>0$, that
\begin{equation}\label{eq:def_lamp}
  \chi_t = x\exp\left(\Tau_{A_{x^{-\alpha}t}}\right),\quad t\geq 0,
\end{equation}
where $A_t = \inf\{s>0;\: \int_{0}^{s}\exp(\alpha \Tau_r)dr>t\}$. Here $\Tau$ is a subordinator, that is a non-decreasing stochastic process with stationary and independent increments and c\`adl\`ag sample paths, and thus its law is characterized by the Bernstein function $\phi(u)= -\log \E[e^{-u \Tau_1}], u\geq 0$, which in this case, for sake of convenience in the later discussion, is expressed for any $ u\geq0$, as
\begin{equation}\label{LKs}
\phi(u)=bu + u \int_0^{1}\ r^{u-1}m(r)dr= bu + u \int_0^{\infty}\ e^{-uy} m(e^{-y})dy
\end{equation}
where  $b \geq 0$ and $r \mapsto m(r)$ is a non-decreasing function on $(0,1)$ and $\int_0^1 (-\ln r \wedge 1)rm(dr)<+\infty$, where $m(r)=\int_0^r m(ds)$, $r \in (0,1)$. Note that under this condition, the mapping $y \mapsto m(e^{-y})$ defined on $\R_+$, is the tail of a L\'evy measure of a subordinator. We also mention that $\phi$ as a Bernstein function is infinitely continuously differentiable and its derivative $\phi'$ is completely monotone, i.e.~ for all  $n\geq 0$, $(-1)^n\frac{d^n}{du^n}\phi(u)\geq0, u>0$, and refer to the monograph \cite{SchillingSongVondracek10} for a thorough account on this set of functions.   Furthermore, to ensure that $\chi$ can be started from $0$, we  assume further that
\begin{equation}\label{eq:mean}
 \E[\Tau_1] = \phi'(0^+) =b+ \int_{0}^{1}\frac{m(r)}{r}dr<+\infty
\end{equation}
see \cite[Theorem 1]{Bertoin-Cab}.
Then, we denote the set of Bernstein functions that satisfy this condition by
\begin{equation*}
\B = \{ \phi \text{ of the form} ~\eqref{LKs} \text{ such that } \phi'(0^+)<+\infty\}.
\end{equation*}
We shall also need, for any $\phi \in \B$, the constant
\begin{equation}\label{eq:def_a}
  \mathfrak{a}_{\phi}=  \sup \{u \leq 0; |\phi(u)|=\infty\} \in (-\infty,0].
\end{equation}
Note that by \cite[Theorem 25.17]{sato1999levy} and after performing an integration by parts, we have that $\int_0^{A}\ r^{\mathfrak{a}_{\phi}+\epsilon-1}m(r)dr<\infty$ for some $A \in (0,1)$ and any $\epsilon>0$. Moreover,  the same result also yields that  $\phi$ admits an analytical extension to  the half-plane $\{z\in \mathbb{C}; \: \Re(z)>\mathfrak{a}_{\phi}\}$.
Next, we  recall from ~\cite[Theorem 6.1]{lamperti1972semi} that the characteristic operator of $\chi$ is given for at least functions $f$ such that $f, tf' \in C_b(\R_+)$, the space of continuous and bounded functions on $\R_+$, by
\begin{equation}\label{eq:boldA}
\mathbf{A} f(t) = t^{-\alpha}\left(  b t f'(t) +  \int_0^{\infty} (f(te^{y})-f(t)) m(d e^{-y}) \right)
\end{equation}
 where $m(d e^{-y})$ stands for the image of the measure $m(dy)$ by the mapping $y\mapsto e^{-y}$.
Next, since $\chi$ has a.s.~ non-decreasing sample paths, this entails that the paths of $\zeta$, as its right-inverse, are a.s.~ non-decreasing. Moreover, they are continuous if and only if the ones of $\chi$ are a.s.~ increasing which from the Lamperti mapping in ~\eqref{eq:def_lamp} is equivalent to $\Tau$ being a.s.~ increasing. This is well known, see e.g.~\cite[Section 5]{kyprianou2014fluctuations}, to be the case when the latter is not a compound Poisson process, that is when
\begin{equation}\label{eq:ass}
\phi(\infty)=\infty \Longleftrightarrow  b>0 \textrm{ or } \int_{A}^{1}  m(dr)=\infty \textrm{ for some } A \in (0,1).
\end{equation}

We also define a subset of $\B$ which will be useful in the sequel, as follows
\begin{equation*}
\BD = \{ \phi \in \B;\ \mathfrak{a}_{\phi} \le -\alpha \text{ and } \lim_{u \downarrow 0} u \phi(u-\alpha)\leq 0 \}.
\end{equation*}
Note that if $\mathfrak{a}_{\phi}<-\alpha$, then we always have $\lim_{u \downarrow 0} u \phi(u-\alpha)=0$.
We refer to the monograph ~\cite{kyprianou2014fluctuations} for a nice account on L\'evy processes.
Now, for any $\phi \in \B$ we consider the function $W_\phi$ which is the unique positive-definite function, i.e.~the Mellin transform of a positive measure, that solves the functional equation, for $\Re(z)>\mathfrak{a}_{\phi}$,
\begin{equation}\label{eq:functional-equation-for-W_phi}
W_\phi(z+1) = \phi(z)W_\phi(z), \quad W_\phi(1) = 1.
\end{equation}
It is easily checked that for any integer $n$, $W_\phi(n+1)=\prod_{k=1}^{n}\phi(k)$, see \cite{Patie-Savov-Bern} for a thorough study of this functional equation.  Throughout, for a random variable $X$, we use  the notation
\[\mathcal{M}_{X}(z)=\E[X^{ z}]\]
for at least any $z\in i\R$, the imaginary line, meaning that $\mathcal{M}_{X}(z-1)$ is  its Mellin transform. 
Next, we recall that for any  integrable function $f$ on $(0,\infty)$, its Mellin transform is defined by
\begin{equation*}
\widehat{f}(z) = \int_0^{\infty}  q^{z-1} f(q)dq
\end{equation*}
for any complex $z$ such that this integrable is finite.
We also recall that  $\chi$ is the Lamperti process of index $\alpha \in (0,1)$ associated to the Bernstein function $\phi \in \B$, and we denote by $\zeta=(\zeta_t)_{t\geq 0}$ its right-inverse, see
~\eqref{eq:def_inv}. We recall that $\zeta$ was used in \cite{semiMarkov_LPS} as a time changed of self-similar Markov processes in the investigation of their extinction time. We are now ready to gather some substantial properties of $\zeta$.

\begin{proposition}\label{prop:L_moments}
Let  $\phi \in \B$, $\alpha \in (0,1)$, and write, for any $u\geq0$,  $\phi_{\alpha}(u)=\phi(\alpha u) \in \B$.
\begin{enumerate}[(i)]
\item\label{it1} For any $t>0$ and $z\in \mathbb C$,
\begin{equation}\label{eq:mom-lamb}
\mathcal{M}_{\zeta_t}(z)=\frac{t^{z\alpha}}{\phi_{\alpha}'(0^+)} \frac{\Gamma(z)}{W_{\phi_{\alpha}}(z)}.
\end{equation}
 In particular, for any $t>0$, $z \mapsto \mathcal{M}_{\zeta_t}(z)$ in analytical on the half-plane $\Re(z)> \mathfrak{a}_{\phi}  = \inf \{u >-1; |\phi(u)|=\infty\} \in (-1,0]$.
\item  $\zeta$ is $\frac{1}{\alpha}$-self-similar and in particular, for all  $q,t>0$ $\E[e^{-q \zeta_t}]=\E[e^{-qt^{\alpha} \zeta_1}]$. Moreover, for any $|q|<\phi(\infty)$,
\begin{equation}\label{eq:mom-lamb}
\E[e^{-q \zeta_1}] = \mathcal{E}_{\phi_{\alpha}}(e^{i\pi}q)=\frac{1}{\phi_{\alpha}'(0^+)}\sum_{n=0}^{\infty}(-1)^n\frac{q^n}{nW_{\phi_{\alpha}}(n)}
\end{equation}
where $\mathcal{E}_{\phi_{\alpha}}$ extends to an analytical function on $\mathbb D_{\phi(\infty)}=\{z\in \mathbb{C}; |z|<\phi(\infty)\}$. Consequently, the law of $\zeta_t$ is, for all $t>0$, moment determinate. Moreover, as a Laplace transform of a Radon measure, the mapping $q \mapsto \mathcal{E}_{\phi_{\alpha}}(e^{i\pi}q)$ is, when $\phi(\infty)=\infty$, completely monotone.
\item\label{it3} Furthermore, if $\mathfrak{a}_{\phi_{\alpha}}<0$, then
$\mathcal{E}_{\phi_{\alpha}}$ admits the following Mellin-Barnes integral representation, for any $0<a<|\mathfrak{a}_{\phi_{\alpha}}|$,
\begin{equation} \label{eq:etaMB}
\mathcal{E}_{\phi_{\alpha}}(z) = -\frac{1}{\phi_{\alpha}'(0^+)} \frac{1}{2\pi i} \int_{a - i \infty}^{a + i \infty}\frac{\phi_{\alpha}(-\xi)}{ \xi}\frac{\Gamma(-\xi)\Gamma(1-\xi)}{W_{\phi_{\alpha}}(1-\xi)}(-z)^{-\xi} d\xi
\end{equation}
which is absolutely convergent (at least) on the sector $\{z \in \mathbb{C}; |\arg(-z)|<\frac{\pi}{2} \}$ and  $q \mapsto \mathcal{E}_{\phi_{\alpha}}(e^{i\pi}q)  \in C^{\infty}_0(\R_+)$, the space of infinitely continuously differentiable functions on $\R_+$ vanishing at infinity along with their derivatives. 
\item Finally, assume that $\phi_{\alpha}$ is meromorphic on the half-plane $\Re(z)>-\p-\epsilon$ for some $\epsilon>0$ with a unique and simple pole at $-\p$. If $\p \in \N$ (resp.~$\p \not\in \N$) and $0<\left| \lim_{z \rightarrow 0} \prod_{k=0}^{\p}\phi_{\alpha}(z-k) \right| < \infty$ (resp.~$0 < |\mathcal{C}_{\p}| < \infty$ where $\mathcal{C}_{\p}=\lim_{z \rightarrow \p} (z-\p)\phi_{\alpha}(-z)$), then, writing $C(\p) = \frac{(-1)^{\p}}{\p W_{\phi_{\alpha}}(-\p)}$ (resp.~ $C(\p) =\frac{\Gamma(\p)\Gamma(-\p)}{W_{\phi_{\alpha}}(1-\p)} \mathcal{C}_{\p}$),
\begin{equation*}
\mathcal{E}_{\phi_{\alpha}}(e^{i\pi} q) \stackrel{+\infty}{\sim} \frac{C(\p)}{ \phi_{\alpha}'(0^+)} q^{-\p},
\end{equation*}
where  for two functions $f$ and $g$ we write $f \stackrel{a}{\sim} g$ if $\lim_{x \rightarrow a}\frac{f(x)}{g(x)}=1$.
\end{enumerate}
\end{proposition}

We proceed by showing that the class of functions $\mathcal{E}_{\phi_{\alpha}}$, which is in bijection with the set of Bernstein functions $\B$, encompasses some famous special functions such as the Mittag-Leffler one and some $q$-series.
\begin{example}[Mittag-Leffler function]\label{rem:stable_sub}
It turns out that  the function $\mathcal{E}_{\phi_{\alpha}}$ is a generalization of the Mittag-Leffler function.
Indeed, recall that $\alpha \in (0,1)$ and define
\begin{equation*}
\phi_{\alpha}(z) = \frac{\Gamma(\alpha+\alpha z)}{\Gamma(\alpha z)}, \quad \mathfrak{R}(z)>-1,
\end{equation*}
which is a Bernstein function,
see e.g.~Loeffen et~al.~\cite{semiMarkov_LPS}. Furthermore, since $\phi'(0^+) = \Gamma(\alpha)<\infty$, we have $\phi \in \B$. Then, an easy algebra yields that $W_{\phi_{\alpha}}(z) = \frac{\Gamma(\alpha z)}{\Gamma(\alpha)}$, $\Re(z)>0$, with $W_{\phi_{\alpha}}(1)=1$.
Therefore, by means of Proposition~\ref{prop:L_moments} and the recurrence relation of the gamma function, one gets, for $q \in \R$ and $t>0$,
\begin{equation*}
\mathcal{E}_{\phi_{\alpha}}(q) = \frac{1}{ \phi_{\alpha}'(0^+)}\sum_{n = 0}^{\infty} \frac{q^n \Gamma(\alpha)}{n \Gamma(\alpha n)} = \sum_{n = 0}^{\infty}\frac{q^n}{\Gamma(\alpha n+ 1)} = \mathcal{E}_{\alpha}(q)
\end{equation*}
where $\mathcal{E}_{\alpha}$ is the Mittag-Leffler function.
Next, since  $\mathfrak{a}_{\phi_{\alpha}}=-1$, 
 \eqref{eq:etaMB} yields  that $\mathcal{E}_{\alpha}$ admits the following Mellin-Barnes integral representation, for any $0<a<1$,
\begin{eqnarray*}
\mathcal{E}_{\alpha}(z) &=& -\frac{1}{\phi_{\alpha}'(0^+)} \frac{1}{2\pi i} \int_{a - i \infty}^{a + i \infty}\frac{\phi_{\alpha}(\xi)}{ \xi}\frac{\Gamma(\xi)\Gamma(1-\xi)}{W_{\phi_{\alpha}}(1-\xi)}(-z)^\xi d\xi \\
&=& -\frac{1}{2\pi i} \int_{a - i \infty}^{a + i \infty} \frac{\Gamma(\xi)\Gamma(1-\xi)}{\Gamma(1-\alpha \xi)} (-z)^\xi d\xi
\end{eqnarray*}
where we use the Stirling formula of the gamma function, recalled  in ~\eqref{eq:gamma_stir} below, to obtain that this integral is absolutely convergent on the sector $\{z \in \mathbb{C}; |\arg z|<(2-\alpha)\frac{\pi}{2} \}$. Next, since the gamma function is a meromorphic function with simple poles at the non-positive integers and $z \mapsto 1/\Gamma(z)$ is an entire function, we have that $\phi_{\alpha}$ has a pole at $-1$ and it is meromorphic on $\mathfrak{R}(z)>-1-\epsilon$ for some $\epsilon>0$. Furthermore,
$
0< \left| \lim_{z \rightarrow 0} \phi_{\alpha}(z) \phi_{\alpha}(z-1) \right| = \left| \lim_{z \rightarrow 0} \frac{\Gamma(\alpha z + \alpha)}{\Gamma(\alpha z - \alpha)} \right| = \left| \frac{\Gamma(\alpha)}{\Gamma(-\alpha)} \right| < \infty.$
Thus, the conditions of Proposition~\ref{prop:L_moments} are satisfied with $\p=1$, and it yields that for any $q,t > 0$,
\begin{equation*}
\mathcal{E}_{\phi_{\alpha}}(e^{i\pi}q)\stackrel{+\infty}{\sim} \frac{q^{-1}}{\Gamma(1-\alpha)}
\end{equation*}
which is the well-known asymptotic behavior of the Mittag-Leffler function, see e.g.~Gorenflo et~al.~\cite[Chapter 3]{mittag_leffler}.
\end{example}

\begin{example}[$\mathfrak q$-series]\label{rem:qseries}
 Let now $\phi$ be the Laplace exponent of a Poisson process of parameter $\log \mathfrak q$, $0<\mathfrak q<1$, that is  $\phi(u)=1-\mathfrak q^u, u\geq 0$, which admits an extension as an entire function. Next, introducing the following  notation from the $\mathfrak q$-calculus, $(a;\mathfrak q)_n=\prod_{k=0}^{n-1} (1-a\mathfrak q^k)$, see \cite{q-series}, and observing that $W_{\phi_{\alpha}}(n+1) = (\mathfrak q^{\alpha};\mathfrak q^{\alpha})_n$, with $n$ an integer, we get that, for $|z|<\phi(\infty)=1$,
 \begin{equation*}
\mathcal{E}_{\phi_{\alpha}}(z) = \frac{1}{\alpha |\ln \mathfrak q|}\sum_{n = 0}^{\infty}\frac{1-\mathfrak q^n}{n} \frac{z^n}{(\mathfrak q^{\alpha};\mathfrak q^{\alpha})_n}.
\end{equation*}
  \end{example}

\begin{proof}
For any bounded Borelian function $f$, we have
\begin{eqnarray}
\E[f(\zeta_t)] &=& \E[f(t^{\alpha}\zeta_1)] = \int_{0}^{\infty}f(t^{\alpha}s)\Prob(\zeta_1 \in ds) \nonumber \\
&=&  \alpha\int_{0}^{\infty}s^{-\alpha-1}f(t^{\alpha}s)\Prob(\chi_1 \in ds^{-\alpha}) \nonumber \\
  &=&\int_{0}^{\infty}f((t/u)^{\alpha})\Prob(\chi_1 \in du) = \E[f\left(t^{\alpha}\chi_1^{-\alpha}\right)] \label{eq:self_sim}
\end{eqnarray}
where we used the identities $\Prob(\zeta_1\leq s)=\Prob(\chi_s \geq 1)=\Prob(\chi_1 \geq s^{-\alpha})$. Then, according  to \cite[Theorem 2.24]{Patie-Savov-Bern}, we deduce  that for any $\Re(z)>0$,
\begin{eqnarray}\label{eq:L^z}
\mathcal{M}_{\zeta_t}(z) = \E[\zeta^{ z}_t] &=& t^{\alpha z} \E[\chi_1^{-z\alpha}] =\frac{t^{\alpha z}}{ \phi_{\alpha}'(0^+)} \frac{\Gamma(z)}{W_{\phi_{\alpha}}(z)}.
\end{eqnarray}
Therefore, in particular, $z \mapsto \mathcal{M}_{\zeta_t}(z)$ is analytical on $\Re(z)>\mathfrak{a}_{\phi_{\alpha}}$, since using ~\eqref{eq:functional-equation-for-W_phi} and the recurrence property of the gamma function, we have
\begin{equation*}
\frac{\Gamma(z)}{\alpha W_{\phi_{\alpha}}(z)} = \frac{\Gamma(z+1)}{W_{\phi_{\alpha}}(z+1)} \frac{\phi_{\alpha}(z)}{\alpha z}
\end{equation*}
and $\lim_{u \downarrow 0}\frac{\phi_{\alpha}(u)}{\alpha u} = \phi'(0^+)<\infty$.
Next, by an expansion of  the exponential function combined with an application of a standard Fubini argument, the identity ~\eqref{eq:L^z} and the recurrence relation for the gamma function, one gets
\begin{equation*}
 \E\left[e^{q \zeta_1}\right]=\sum_{n=0}^{\infty}\E[\zeta^{ n}_1]\frac{q^n}{n!}=\frac{1}{\phi_{\alpha }'(0^+)}\sum_{n=0}^{\infty} \frac{1}{n}\frac{q^n}{W_{\phi_{\alpha}}(n)} = \mathcal{E}_{\phi_{\alpha}}(q)
\end{equation*}
where, by using the functional equation \eqref{eq:functional-equation-for-W_phi}, the series is easily checked to be absolutely convergent, and hence an analytical function, on $\{z\in \mathbb{C}; \: |z|<\phi(\infty)\}$. Then, admitting exponential moments, the law of $\zeta_t$ is moment-determinate for all $t>0$. Next, since $\chi$ is an $\alpha$-self-similar process, by ~\eqref{eq:def_inv}, plainly $\zeta$ is $\frac{1}{\alpha}$-self-similar.
To derive the Mellin-Barnes integral representation of $\Ep$, we first observe from ~\eqref{eq:L^z} that the mapping
\[ z \mapsto \E[\zeta_1^z] = \frac{1}{\phi_{\alpha }'(0^+)} \frac{\Gamma(z)}{W_{\phi_{\alpha}}(z)} \]
 is analytical on $\Re(z)>0$ since $z \mapsto \Gamma(z)$ and $z \mapsto W_{\phi_{\alpha}}(z)$ are analytical on $\Re(z)>0$,  and the latter is also zero-free on the same half-plane, see \cite[Theorem 4.1]{Patie-Savov-Bern}. Next, let us assume that $\mathfrak{a}_{\phi_{\alpha}} <0$, and observe,  using ~\eqref{eq:functional-equation-for-W_phi}, that
\begin{equation*}
\int_0^{\infty}\E\left[e^{-q \zeta_1}\right] q^{\xi-1}dq  = \E\left[\zeta_1^{-\xi}\right] \Gamma(\xi) = \frac{\Gamma(-\xi)}{\phi_{\alpha }'(0^+) W_{\phi_{\alpha}}(-\xi)}\Gamma(\xi)= \frac{1}{\phi_{\alpha }'(0^+)} \frac{\phi_{\alpha}(-\xi)}{-\xi} \frac{\Gamma(\xi)\Gamma(1-\xi)}{W_{\phi_{\alpha}}(1-\xi)}
\end{equation*}
which is analytical on $0<\Re(\xi)<|\mathfrak{a}_{\phi_{\alpha}}|$. Indeed, first, since $\xi \mapsto \Gamma(\xi)$ is analytical on the right half-plane $\Re(\xi)>0$, plainly,  $\xi \mapsto \Gamma(\xi)$ is analytical on $ \Re(\xi)>0 $. Next, as above, we have that $\xi \mapsto W_{\phi_{\alpha}}(1-\xi)$ is  zero-free on $\Re(\xi)<1$ and meromorphic on $\Re(\xi)<|\mathfrak{a}_{\phi_{\alpha}}|$ with a simple pole at $1$, see \cite[Theorem 4.1]{Patie-Savov-Bern}, which cancels the one of $\Gamma(1-\xi)$, and we get the sought analyticity from the definition of $\mathfrak{a}_{\phi_{\alpha}}$.
We write  
\begin{equation}\label{eq:Ehat*}
\Ehat^*(\xi) =\frac{1}{\phi_{\alpha }'(0^+)} \frac{\phi_{\alpha}(-\xi)}{-\xi} \frac{\Gamma(\xi)\Gamma(1-\xi)}{W_{\phi_{\alpha}}(1-\xi)}.
\end{equation}
Next,  we recall that the Stirling's formula yields that for any $a \in \R$ fixed, when $|b| \rightarrow \infty$,
\begin{equation}\label{eq:gamma_stir}
|\Gamma(a+ib)| \stackrel{\infty}{\sim} C_a |b|^{a-\frac{1}{2}} e^{-|b|\frac{\pi}{2}}
\end{equation}
where $C_a > 0$, see e.g.~\cite[Lemma 9.4]{patie2015spectral}. Furthermore, ~\cite[Proposition 6.12(2)]{patie2015spectral} gives that for any $a > 0$,
\begin{equation}\label{eq:W_phi_bound}
\overline{\lim}_{|b| \rightarrow \infty} \frac{e^{-|b|\frac{\pi}{2}} |b|^{-\frac{1}{2}}}{|W_{\phi_{\alpha}}(a+bi)|}  \le c_+(a)
\end{equation}
for some positive finite constant $c_+(a)$. Therefore, taking $\xi = a+ ib$ for any $b \in \R$ and $0<a<|\mathfrak{a}_{\phi_{\alpha}}|$, using ~\eqref{eq:gamma_stir} and ~\eqref{eq:W_phi_bound}, there exists $\tilde{C}_a>0$ such that for $a$ fixed and $|b|$ large
\begin{eqnarray}\label{eq:f-hat-stirling}
|\Ehat^*(\xi)| = \left| \frac{\phi_{\alpha}(-\xi)}{-\xi}\frac{\Gamma(\xi)\Gamma(1-\xi)}{\phi_{\alpha }'(0^+)W_{\phi_{\alpha}}(1-\xi)}\right| \le \tilde{C}_a \ |b|^{2a-\frac{1}{2}} e^{-|b|\frac{\pi}{2}}
\end{eqnarray}
where we used the upper bound of $\phi$ found in \cite[Proposition 6.2]{Patie-Savov-Bern}.
 Thus, by Mellin's inversion formula, see e.g.~ \cite[Chapter 11]{mellin_tr_misra_lavoine}, one gets that for any $0<a<|\mathfrak{a}_{\phi_{\alpha}}|$,
\begin{equation*}
\E\left[e^{-z \zeta_1}\right] = \frac{1}{2\pi i} \int_{a-i \infty}^{a+ i \infty}\Ehat^*(\xi)z^{-\xi} d\xi
\end{equation*}
and thus by uniqueness of analytical extension, we get that
\begin{equation}\label{eq:f}
\Ep(z)= \frac{1}{\phi_{\alpha }'(0^+)} \frac{1}{2\pi i} \int_{a - i \infty}^{a + i \infty}\frac{\phi_{\alpha}(-\xi)}{-\xi}\frac{\Gamma(\xi)\Gamma(1-\xi)}{W_{\phi_{\alpha}}(1-\xi)}(-z)^{-\xi} d\xi
\end{equation}
 which is a function analytical on the sector $\{z \in \mathbb{C}; \ |\arg(-z)|<\frac{\pi}{2}\}$.
Indeed, first, by the discussion above, we have that $\xi\mapsto \Ehat^*(\xi)$ is analytical on the strip $0<\Re(\xi)<|\mathfrak{a}_{\phi_{\alpha}}|$.
Next, taking $\xi=a+ib$, using ~\eqref{eq:functional-equation-for-W_phi} and ~\eqref{eq:f-hat-stirling}, we have that when $|b|$ is large, there exists a constant $\tilde{C}_a > 0$ such that
\begin{equation}\label{eq:eta_hat_bound}
\left| \Ehat^*(\xi) (-z)^{-\xi} \right| \le  \tilde{C}_a\ |z|^{-a} \ |b|^{2a-\frac{1}{2}} e^{-|b|\frac{\pi}{2}+b \arg (-z)}.
\end{equation}
Putting pieces together, we indeed get the claimed analytical property of $\Ep$.
Now, to study the asymptotic behavior of $\Ep$, we write, for $q>0$,
\begin{equation}\label{eq:MB}
\Ep(e^{i\pi}q) = \frac{1}{\phi_{\alpha }'(0^+)} \frac{1}{2\pi i} \int_{a-i \infty}^{a+i \infty} \frac{\Gamma(\xi)\Gamma(-\xi)}{W_{\phi_{\alpha}}(-\xi)}q^{-\xi} d\xi,
\end{equation}
recall that the gamma function has simple poles at non-positive integers, and investigate the poles of $f_{\phi}(\xi) = \frac{\Gamma(\xi)}{W_{\phi_{\alpha}}(\xi)}$. Using ~\eqref{eq:functional-equation-for-W_phi}, we get that $f_{\phi}$ satisfies to the following functional equation
\begin{equation} \label{eq:def_f}
f_{\phi}(\xi+1) = \frac{\xi}{\phi_{\alpha}(\xi)} f_{\phi}(\xi).
\end{equation}
Next, since $0<\phi'(0^+) < \infty$, we have that $0<\lim_{\xi\rightarrow 0} \frac{\xi}{\phi_{\alpha}(\xi)} < \infty$. Moreover, since $q\mapsto\frac{1}{\phi_{\alpha}(q)}$ is the Laplace transform of a positive measure whose support is contained in $[0,\infty)$, see e.g.~\cite[Proposition 4.1(4)]{patie2015spectral}, it has its singularities on the negative real line. Thus, $\mathfrak{s}<0$ is a pole for $f_{\phi}$ if $\phi_{\alpha}(\mathfrak{s}) = -\infty$. Next, since $\phi_{\alpha}$ is meromorphic on $\Re(\xi)>-\p-\epsilon$ with a unique pole at $-\p$ with $\p>0$, we can extend the domain of analiticity of $W_{\phi_{\alpha}}$ on $\Re(\xi)>-\p-\epsilon$, and by Cauchy's theorem, we have
\begin{equation}\label{eq:eta_res}
\Ep(e^{i\pi}q) = \frac{1}{\phi_{\alpha }'(0^+) } Res(\Ep, \p) +  \frac{1}{\phi_{\alpha }'(0^+)} \frac{1}{2\pi i} \int_{\p+\epsilon-i \infty}^{\p+\epsilon +i \infty} \frac{\Gamma(\xi)\Gamma(-\xi)}{W_{\phi_{\alpha}}(-\xi)}q^{-\xi} d\xi.
\end{equation}
Next, if $\p \in \N$, then we have
\begin{equation*}
W_{\phi_{\alpha}}(\xi-\p) = \frac{W_{\phi_{\alpha}}(\xi+1)}{\prod_{k=0}^{\p}\phi_{\alpha}(\xi-k)}, \quad \Re(\xi)>0.
\end{equation*}
Hence, since $W_{\phi_{\alpha}}(1)=1$, we deduce that
\begin{equation*}
0< \lim_{\xi\rightarrow 0}|W_{\phi_{\alpha}}(\xi-\p)| = | W_{\phi_{\alpha}}(-\p)| = \frac{1}{\left| \lim_{\xi \rightarrow 0}\prod_{k=0}^{\p}\phi_{\alpha}(\xi-k)\right| } <\infty,
\end{equation*}
and since $Res(\Gamma, -n) = \frac{(-1)^n}{n!}$, $n = 1,2,\cdots$, we have
\begin{equation*}
Res(\Ep, \p) =\frac{\Gamma(\p)}{W_{\phi_{\alpha}}(-\p)} \frac{(-1)^{\p}}{\p !}q^{-\p} = \frac{(-1)^{\p} }{\p W_{\phi_{\alpha}}(-\p)}q^{-\p}.
\end{equation*}
Therefore, combining this with ~\eqref{eq:eta_res}, we obtain
\begin{equation*}
\Ep(e^{i\pi}q) \stackrel{+\infty}{\sim} \frac{(-1)^{\p}}{\phi_{\alpha }'(0^+)\p W_{\phi_{\alpha}}(-\p)}q^{-\p}.
\end{equation*}
Otherwise, if $\p \not\in \N$, we have
\begin{eqnarray*}
\lim_{\xi \rightarrow \p} (\xi-\p)\frac{\Gamma(-\xi)}{W_{\phi_{\alpha}}(-\xi)} = \Gamma(-\p) \lim_{\xi \rightarrow \p} \frac{(\xi-\p)\phi_{\alpha}(-\xi)}{W_{\phi_{\alpha}}(1-\xi)} = -\frac{1}{\p}\frac{\Gamma(1-\p)}{W_{\phi_{\alpha}}(1-\p)} \lim_{\xi \rightarrow \p}(\xi-\p)\phi_{\alpha}(-\xi)
\end{eqnarray*}
which is finite since $\frac{\Gamma(1-\p)}{W_{\phi_{\alpha}}(1-\p)} < \infty$ as $-\p$ was the first pole of the function $f_{\phi}$ defined in \eqref{eq:def_f}, and by assumption $0< |\mathcal{C}_{\p}|=|\lim_{\xi \rightarrow \p}(\xi-\p)\phi_{\alpha}(-\xi)|<\infty$. Hence,
\begin{equation*}
Res(\Ep,\p) = \frac{\Gamma(\p)\Gamma(-\p)}{W_{\phi_{\alpha}}(1-\p)} \mathcal{C}_{\p}\  q^{-\p}
\end{equation*}
and, with ~\eqref{eq:eta_res}, we get
\begin{equation*}
\Ep(e^{i\pi}q) \stackrel{+\infty}{\sim}  \frac{\Gamma(\p)\Gamma(-\p)}{\phi_{\alpha }'(0^+)W_{\phi_{\alpha}}(1-\p)} \mathcal{C}_{\p}\  q^{-\p}.
\end{equation*}
To conclude the proof,   we use the estimate \eqref{eq:f-hat-stirling} to apply the Riemann-Lebesgue lemma to get
\begin{equation*}
\lim_{q \to \infty} \frac{q^{-\p}}{\phi_{\alpha }'(0^+)} \int_{\p+\epsilon-i \infty}^{\p+\epsilon +i \infty} \frac{\Gamma(\xi)\Gamma(-\xi)}{W_{\phi_{\alpha}}(-\xi)}q^{-\xi} d\xi = \lim_{q \to \infty} q^{-\epsilon} \int_{- \infty}^{ \infty}e^{ib \ln q}\frac{\Gamma(\p+\epsilon+ib)\Gamma(-\p- \epsilon-ib)}{\phi_{\alpha }'(0^+)W_{\phi_{\alpha}}(-(\p+\epsilon+ib))}  db =0.
\end{equation*}
\end{proof}


\section{Self-similar multiplicative convolution generalization of fractional operators}\label{sec:frac_op}
In this section,  we introduce a class of multiplicative convolution operators that generalize the fractional Caputo derivative and provide some interesting properties. In particular, we show that they have the same self-similarity property than the fractional Caputo derivative and we identify conditions under which these operators admit the functions $\Ep$ as eigenfunctions. 
Inspired by the multiplicative convolution representation of the fractional Caputo derivative presented in \eqref{eq:Cap}, we introduce its generalization as follows. We denote by $AC[0,t]$ the space of absolutely continuous functions on $[0,t]$, $t>0$, and by $L^1(0,t)$ the space of Lebesgue integrable functions on $(0,t)$, $t>0$.
\begin{definition}
\begin{enumerate}[1)]
\item Let $m$ be a non-negative measurable function defined on $(0,1)$, $b \in \R$ and write $\Phi(z) = bz + z \int_0^1 r^{z-1}m(r)dr$ for $z \in \mathbb{C}_{\Phi} =\{z\in \mathbb{C};r \mapsto r^{z-1}m(r) \in L^1(0,1) \}$. For $\alpha \in (0,1)$ and $f \in \calD(\opPhi)=C^1(\R_+)\cap\lbrace f \in AC[0,t]; y\mapsto f'(y)m\left( \frac{y}{t} \right) \in L^1(0,t) \rbrace$, we define
\begin{equation} \label{eq:defDp}
\opPhi f(t) = t^{1-\alpha}  b f'(t) + t^{-\alpha} f' \star m(t)
\end{equation}
where we recall that $f' \star m(t)= \int_0^{t} f'(r) m\left(\frac{r}{t} \right)dr$.
\item If $\phi \in \B$ is defined by~\eqref{LKs}, $b \geq 0$ and $r \mapsto m(r)$ is a non-decreasing function on $(0,1)$ such that $\int_0^1 (-\ln r \wedge 1)rm(dr)<+\infty$, then $\Phi \equiv \phi$ and we write $\opPhi = \opphi$.
\end{enumerate}
\end{definition}
We proceed by providing some substantial properties of these generalized fractional operators.

\begin{proposition}\label{prop:operator}
\begin{enumerate}[(i)]
\item \label{item1} $\opPhi$ is a linear operator that satisfies the scaling property
\begin{equation*}
\opPhi d_cf (t) = c^{\alpha} \opPhi f (c t), \quad c,t>0.
\end{equation*}
\item \label{item2} For any $z \in \mathbb{C}_{\Phi}$ and $t>0$, writing $p_z(t)=t^z$, we have
\begin{equation}
\opPhi p_z(t) = \Phi(z) p_{z-\alpha}(t).
\end{equation}
Consequently, if $\phi \in \B$, then for any $z \in \mathbb C_{(\mathfrak{a}_{\phi},\infty)}$, we have $\opphi p_z(t) = \phi(z) p_{z-\alpha}(t)$. Moreover, let $m_{\alpha}(r)=r^{-\alpha}m(r)$, $r \in (0,1)$ and for $z \in \mathbb{C}_{\bPhi} =\{z\in \mathbb{C};r \mapsto r^{z-1}m_{\alpha}(r) \in L^1(0,1) \}$, define
\begin{equation}\label{eq:bPhi}
\bPhi(z) = \frac{z}{z-\alpha}\Phi(z-\alpha).
\end{equation}
Then, for $z \in \mathbb{C}_{\bPhi}$ and $t>0$,
\begin{equation}
\opbPhi p_z(t) = \bPhi(z) p_{z-\alpha}(t).
\end{equation}
Note that, in any case, $\opbPhi p_0(t)=0$.
\item \label{item4} Assume that either $\phi \in \B$ with $\mathfrak{a}_{\phi}<0$ and $\phi(\infty)=\infty$ or $\phi \in \BD$ with $\mathfrak{a}_{\phi} < -\alpha$.
Then, writing $F_q(t)=\mathcal{E}_{\phi_{\alpha}}(qt^{\alpha})$, we have, for any $q \in \R$ and $t >0$,
\begin{equation}\label{eq:d_qE}
\opbPhi F_q(t) = q F_q(t)
\end{equation}
where, as in \eqref{eq:bPhi}, we have set $\bPhi(z)= \frac{z}{z-\alpha}\phi(z-\alpha)$.
Moreover, if in addition $\phi \in \BD$, see \eqref{LKs}, and $r \mapsto m_{\alpha}(r)=r^{-\alpha}m(r)$ is a non-decreasing function on $(0,1)$,
then the mapping $\bPhi$ is a Bernstein function, and $\bPhi \in \B$ if $\mathfrak{a}_{\phi} < -\alpha$.
\item \label{item3} Let $\phi \in \B$. Then, we have the following  relation,  at least for functions $f$ such that $f, tf' \in C_b(\R_+)$,
\begin{equation}
\opphi \Lambda f(t)  =-t^{-2\alpha} \Lambda \mathbf{A}  f \left(t \right),
\end{equation}
where $\Lambda f = f \circ \iota$ with  $\iota(y) = \frac{1}{y}$, is an involution, and $\mathbf{A}$ is the characteristic operator, defined in \eqref{eq:boldA}, of the self-similar Markov process associated  via the Lamperti mapping with $\phi$.
\end{enumerate}
\end{proposition}

\begin{remark}
Note that if  $\phi \in \B$ with $\phi(\infty)<\infty$, then \eqref{eq:d_qE} still holds for any $q \in \R$ and $t > 0$, such that $|q|t^{\alpha}<\phi(\infty)$.
\end{remark}

\begin{example}
Let $\chi$ be an $\alpha$-stable subordinator, and note that it is also an increasing positive self-similar Markov process. Moreover, the Laplace exponent of the subordinator associated with $\chi$, via the Lamperti mapping, is well known to be $\phi(u) = \frac{\Gamma(u+\alpha)}{\Gamma(u)}$, $u>0$, see e.g.~\cite{semiMarkov_LPS}, and note that in this case $\mathfrak{a}_{\phi}= -\alpha$ with $\lim_{u \downarrow 0} u \phi(u-\alpha)=\frac{1}{\Gamma(-\alpha)}<0$. Using the integral representation for the ratio of two gamma functions, see e.g.~\cite[(15)]{gammaratio}, we can write $\phi$ as
\begin{equation*}
\phi(u) =  \frac{\alpha}{\Gamma(1-\alpha)} \int_0^{\infty} (1-e^{- u y}) \frac{e^{-\alpha y}}{(1-e^{-y})^{\alpha+1}}dy,
\end{equation*}
from where we deduce,  since $m(e^{-y})$ is  the tail of a L\'evy measure,  that, for any $y>0$,
\begin{eqnarray*}
m(e^{-y}) &=& \frac{\alpha}{\Gamma(1-\alpha)}\int_y^{\infty}  \frac{e^{-\alpha r}}{(1-e^{-{r}})^{\alpha+1}}dr= \frac{\alpha}{\Gamma(1-\alpha)}\int_y^{\infty}  \frac{e^r}{(e^r-1)^{\alpha+1}}dr =\frac{(e^y-1)^{-\alpha}}{\Gamma(1-\alpha)}.
\end{eqnarray*}
Hence, we have, for any $r \in (0,1)$,
\begin{equation*}
m(r) =r^{\alpha}\frac{(1-r)^{-\alpha}}{\Gamma(1-\alpha)}.
\end{equation*}
Next, noting that $r \mapsto m_{\alpha}(r)=r^{-\alpha}m(r)=\frac{(1-r)^{-\alpha}}{\Gamma(1-\alpha)}$ is a non-decreasing function on $(0,1)$, item~\eqref{item4} implies that $\bPhi$ is a Bernstein function, and we obtain
\begin{eqnarray*}
\opbPhi f(t) &=&
 \frac{t^{-\alpha}}{\Gamma(1-\alpha)} \int_0^{t} f'(y)\left( \frac{y}{t}\right)^{-\alpha}\left(\frac{t}{y} -1\right)^{-\alpha}dr =
\frac{~^C d^{\alpha}}{dt^{\alpha}}f(t).
\end{eqnarray*}

\end{example}

\begin{example}
Let $\phi(u)=1-\mathfrak{q}^{u}$, $u \geq 0$, be as in Example \ref{rem:qseries} with $0<\mathfrak{q}<1$. Then, we can write
\begin{equation*}
\phi(u) = \int_0^{\infty} (1-e^{-uy}) \delta_{-\log \mathfrak{q}}(y),
\end{equation*}
where $\delta_{-\log \mathfrak{q}}$ is the Dirac measure supported on $\{ -\log \mathfrak{q} \}$. Therefore, as above, we deduce that, for any $y \geq 0$,
\begin{equation*}
m(e^{-y}) = \int_y^{\infty}  \delta_{-\log \mathfrak{q}}(r) = \mathbbm{1}_{\{ y \le -\log \mathfrak{q} \}}.
\end{equation*}
Thus, a change of variable yields, that for any $r \in (0,1)$,
\begin{equation*}
m(r) = \mathbbm{1}_{\{\mathfrak{q} \leq  r < 1  \}}.
\end{equation*}
Therefore, since $r \mapsto r^{-\alpha} m(r)$ is a non-decreasing function, item \eqref{item4} implies that $\bPhi$ is a Bernstein function, and we have
\begin{eqnarray*}
\opbPhi f(t) &=& t^{-\alpha}\int_0^{t} f'(y) \left( \frac{y}{t} \right)^{-\alpha} \mathbbm{1}_{\{ y \geq t\mathfrak{q} \}} dy = \int_{t\mathfrak{q}}^{t} f'(y) y^{-\alpha} dy \\
&=& t^{-\alpha} f(t) - (t\mathfrak{q})^{-\alpha} f(t \mathfrak{q}) + \alpha \int_{t\mathfrak{q}}^t f(y) y^{-\alpha -1}dy
\end{eqnarray*}
where in the last step we performed an integration by parts.
\end{example}

\begin{proof}
First, plainly $\opPhi$ is a linear operator and for $c,t > 0$, we note that
\begin{eqnarray*}
\opPhi d_cf (t) &=& t^{1-\alpha}bcf'(ct) + t^{-\alpha} \int_0^{t} cf'(cy) m\left(\frac{y}{t} \right)dy\\
&=& c^{\alpha}b (ct)^{1-\alpha}f'(ct) +c^{\alpha} (ct)^{-\alpha}  \int_0^{ct} f'(r) m\left(\frac{r}{ct} \right)dr \\
&=& c^{\alpha}\opPhi f (c t)
\end{eqnarray*}
and this completes the proof of item \eqref{item1}. To prove item \eqref{item2}, we perform a change of variable and get
\begin{eqnarray*}
\opPhi p_z(t) &=& t^{1-\alpha}  b  z t^{z-1} + t^{-\alpha} \int_0^t z  y^{z-1} m\left( \frac{y}{t}\right) dy \\
&=& t^{z-\alpha}  b  z + t^{z-\alpha} z \int_0^1 r^{z-1}m(r) dr =  \Phi(z) p_{z-\alpha}(t).
\end{eqnarray*}
We now prove item \eqref{item4}.  We first take $\phi \in \B$ and thus deduce that  the mapping
\begin{equation}\label{eq:phi-boldphi}
z\mapsto \bPhi(z) = \frac{z}{z-\alpha}\phi(z-\alpha) = b z + z \int_0^{1} r^{z-1}m_{\alpha}(r)dr
\end{equation}
is analytical on the right-half plane $\Re(z)> \mathfrak{a}_{\phi}+\alpha$, see below \eqref{eq:def_a}. 
Moreover, from item \eqref{item2}, we  get, for such a $z$, that
\begin{equation}\label{eq:Ap_z}
\opbPhi p_z(t) = \bPhi(z)p_{z-\alpha}(t).
\end{equation}
 Next,  let us assume first that $\phi(\infty)=\infty$ and $\mathfrak{a}_{\phi} <0$. Then,
using the series expansion of $\mathcal{E}_{\phi_{\alpha}}$ in \eqref{eq:mom-lamb} combined with the previous identity \eqref{eq:Ap_z} with $z=\alpha n$, we get, writing $F_q(t)=\mathcal{E}_{\phi_{\alpha}}(qt^{\alpha})$, for any $q \in \R$ and $t>0$,
\begin{eqnarray}
\opbPhi F_q(t) &=& t^{1-\alpha}b F_q'(t) + t^{-\alpha} \int_0^t F_q'(y) m\left( \frac{y}{t}\right)dy \nonumber\\
&=& t^{1-\alpha}b \frac{\partial}{\partial t} \sum_{n=1}^{\infty} \frac{q^n  p_{\alpha n}(t)}{nW_{\phi_{\alpha}}(n)} + t^{-\alpha} \int_0^t \left(\frac{\partial}{\partial y} \sum_{n=0}^{\infty} \frac{q^n  p_{\alpha n}(y)}{nW_{\phi_{\alpha}}(n)}\right)\ m\left( \frac{y}{t}\right)dy \nonumber\\
&=& \frac{1}{ \phi_{\alpha}'(0^+)} \sum_{n=1}^{\infty} \frac{q^n \opbPhi p_{\alpha n}(t)}{nW_{\phi_{\alpha}}(n)} \nonumber \\
&=& \frac{1}{ \phi_{\alpha}'(0^+)} \sum_{n=1}^{\infty} \frac{q^n \bPhi(\alpha n) p_{\alpha(n-1)}(t)}{nW_{\phi_{\alpha}}(n)} \nonumber \\
&=& \frac{1}{ \phi_{\alpha}'(0^+)} \sum_{n=1}^{\infty} \frac{q^n p_{\alpha(n-1)}(t)}{(n-1)W_{\phi_{\alpha}}(n-1)} = q F_q(t) \label{eq:q=1}
\end{eqnarray}
where we used, from \eqref{eq:defDp}, that $\opbPhi p_{0}(t)=0$, the functional equation satisfied by $W_{\phi_{\alpha}}$, see \eqref{eq:functional-equation-for-W_phi}, the relation \eqref{eq:phi-boldphi}, the fact that power series can be term-by-term differentiated inside the interval of its convergence, and changing the order of integration and summation by a dominated convergence argument which is justified since
\begin{eqnarray*}
\sum_{n=1}^{\infty} \frac{|q|^n}{nW_{\phi_{\alpha}}(n)}\int_0^t \alpha n y^{\alpha n -1} \left(\frac{t}{y}\right)^{\alpha} m\left(\frac{y}{t}\right)dy &=& \sum_{n=1}^{\infty}  \frac{|q|^nt^{\alpha n}}{nW_{\phi_{\alpha}}(n)} \alpha n \int_0^1 r^{\alpha n - \alpha - 1}m(r)dr\\
&=&  \sum_{n=1}^{\infty}  \frac{|q|^nt^{\alpha n}}{W_{\phi_{\alpha}}(n)} \alpha  \int_0^1 r^{\alpha n - \alpha - 1}m(r)dr < +\infty
\end{eqnarray*}
 as  by assumption $a_{\phi}<0$, that is $\int_0^1 r^{\alpha n-\alpha -1}m(r)dr <\infty$ for any $n \geq 1$, see below \eqref{eq:def_a} again.
Now, we move to the proof of item $\eqref{item4}$ when   $\phi \in \BD$ with $\mathfrak{a}_{\phi} < -\alpha$. Recall from ~\eqref{eq:etaMB} that one has, for any $0<a<|\mathfrak{a}_{\phi_{\alpha}}|$, $q\in \R$, $t >0$,
\begin{equation*}
F_{q}(t) = \Ep(qt^{\alpha}) = \frac{1}{2\pi i} \int_{a-i \infty}^{a+ i \infty}(-qt^{\alpha})^{-z}\Ehat^*(z)dz = \frac{1}{2\pi i} \int_{a-i \infty}^{a+ i \infty} p_{-\alpha z}(t) (-q)^{-z} \Ehat^*(z)dz.
\end{equation*}
Next, observe that for any $z=a+ib$ with  $|b|$ large,
\begin{equation*}
\left| \frac{\partial}{\partial t} p_{-\alpha z}(t) (-q)^{-z} \Ehat^*(z)\right| \leq \alpha |q|^{-a} t^{-\alpha a -1}  \left|   z\Ehat^*(z) \right|  \le  \tilde{C}_a\ t^{-\alpha a -1}  |q|^{-a} \ |b|^{2a+\frac{1}{2}} e^{-|b|\frac{\pi}{2}}
\end{equation*}
where $\tilde{C}_a>0$ and we used the bound \eqref{eq:eta_hat_bound}. This justifies the application of  the dominated convergence theorem  to get
\begin{equation*}
\frac{\partial}{\partial t}F_q(t) =\frac{\partial}{\partial t} \frac{1}{2\pi i} \int_{a-i \infty}^{a+ i \infty} p_{-\alpha z}(t) (-q)^{-z} \Ehat^*(z)dz = - \frac{\alpha }{2\pi i} \int_{a-i \infty}^{a+ i \infty} z t^{-\alpha z-1} (-q)^{-z} \Ehat^*(z)dz.
\end{equation*}
Thus,  denoting $\opma f(t)= t^{-\alpha}f' \star m_{\alpha}(t)$,  we obtain
\begin{eqnarray*}
\opma F_q(t) 
&=& \frac{t^{-\alpha}}{2\pi i} \int_0^{t} \int_{a-i \infty}^{a+ i \infty} \frac{\partial}{\partial y} p_{-\alpha z}(y) (-q)^{-z} \Ehat^*(z)dz\  m_{\alpha}\left(\frac{y}{t} \right)dy.
\end{eqnarray*}
Now, by \eqref{eq:eta_hat_bound}, we have
\begin{eqnarray*}
\left|  \frac{\partial}{\partial y} p_{-\alpha z}(y) (-q)^{-z} \Ehat^*(z) m_{\alpha}\left(\frac{y}{t} \right)  \right| \le \tilde{C}_a\ y^{-\alpha a -1 }  |q|^{-a} \ |b|^{2a+\frac{1}{2}} e^{-|b|\frac{\pi}{2}} m_{\alpha}\left(\frac{y}{t} \right).
\end{eqnarray*}
Therefore, since $\mathfrak{a}_{\phi} <  -\alpha$ and $0<a<|\mathfrak{a}_{\phi_{\alpha}}|$, one can choose $a = \frac{\epsilon}{\alpha}$ for $\epsilon>0$ such that $\mathfrak{a}_{\phi}< -\alpha - \epsilon$,  to deduce that, for all $t>0$,
\begin{equation}
\int_0^t y^{-\alpha a -1 }  m_{\alpha}\left(\frac{y}{t} \right) dy = t^{\alpha} \int_0^t y^{-(\alpha+\alpha a+1)} m(y)dy <\infty.
\end{equation}
Thus,
\begin{equation*}
\int_0^{t} \int_{a-i \infty}^{a+ i \infty} \tilde{C}_a\ y^{-\alpha a -1 }  |q|^{-a} \ |b|^{2a+\frac{1}{2}} e^{-|b|\frac{\pi}{2}} m_{\alpha}\left(\frac{y}{t} \right) dz\ dy <\infty
\end{equation*}
and,  by Fubini's theorem, we get
\begin{equation*}
\opma F_q(t) = \frac{1}{2\pi i} \int_{a-i \infty}^{a+ i \infty}\opm p_{-\alpha z}(t) (-q)^{-z} \Ehat^*(z)dz.
\end{equation*}
Finally, putting pieces together, we have
\begin{eqnarray*}
\opbPhi F_{q}(t) &=& t^{1-\alpha}\ b\ \frac{\partial}{\partial t} F_q(t) + \opma F_q(t) \\
&=& t^{1-\alpha}\ b\ \frac{1}{2\pi i} \int_{a-i \infty}^{a+ i \infty}  \frac{\partial}{\partial t} p_{-\alpha z}(t) (-q)^{-z} \Ehat^*(z)dz +  \frac{1}{2\pi i} \int_{a-i \infty}^{a+ i \infty} \opma p_{-\alpha z}(t) (-q)^{-z} \Ehat^*(z)dz \\
&=&  \frac{1}{2\pi i} \int_{a-i \infty}^{a+ i \infty}\opbPhi p_{-\alpha z}(t)(-q)^{-z}\Ehat^*(z)dz.
\end{eqnarray*}
Next,  since from~\eqref{eq:Ap_z} we have that $\opbPhi p_{-\alpha z}(t) =  \frac{z}{z+1} \phi_{\alpha}(-z-1)p_{-\alpha (z +1)}(t)$ and  recalling from \eqref{eq:functional-equation-for-W_phi} that for $\Re(z)<\mathfrak{a}_{\phi_{\alpha}}$, $W_{\phi_{\alpha}}(1-z) = \phi_{\alpha}(-z) W_{\phi_{\alpha}}(-z)$, and the recurrence relation of the gamma function, $\Gamma(z+1) = z\Gamma(z)$, $z \in \mathbb{C}$, we obtain
\begin{eqnarray}
\opbPhi F_{q}(t) &=& \frac{1}{2\pi i} \int_{a-i \infty}^{a+ i \infty}\opbPhi p_{-\alpha z}(t)(-q)^{-z}\Ehat^*(z)dz \\
&=& \frac{1}{2\pi i} \int_{a-i \infty}^{a+ i \infty} \frac{z}{z+1}\phi_{\alpha}(-z-1) p_{-\alpha (z +1)}(t) (-q)^{-z}\Ehat^*(z)dz \nonumber \\
&=& - \frac{q}{2\pi i} \int_{a-i \infty}^{a+ i \infty} \frac{z}{z+1}\phi_{\alpha}(-1-z) (-qt^{\alpha})^{-(z+1)} \frac{1}{ \phi_{\alpha}'(0^+)} \frac{\Gamma(z)\Gamma(-z)}{W_{\phi_{\alpha}}(-z)}dz \nonumber \\
&=& \frac{q}{2\pi i} \int_{a-i \infty}^{a+ i \infty} (-qt^{\alpha})^{-(z+1)} \frac{1}{ \phi_{\alpha}'(0^+)} \frac{\Gamma(z+1)\Gamma(-z-1)}{W_{\phi_{\alpha}}(-z-1)}dz \nonumber \\
&=&  \frac{q}{2\pi i} \int_{a+1-i \infty}^{a+1+ i \infty} (-qt^{\alpha})^{-z} \frac{1}{ \phi_{\alpha}'(0^+)} \frac{\Gamma(z)\Gamma(-z)}{W_{\phi_{\alpha}}(-z)}dz \nonumber \\
&=&  \frac{q}{2\pi i} \int_{a+1-i \infty}^{a+1+ i \infty} (-qt^{\alpha})^{-z}\Ehat^*(z)dz = q F_{q}(t) \label{eq:-q}
\end{eqnarray}
where the justification of the last identity is given as follows. First, the mapping $ z\mapsto F(z) = p_{-\alpha z}(t)(-q)^{-z}\Ehat^*(z)$ is analytical in the strip $\Re(z) \in (a,a+1)$, and for some $b > 0$, we have
\begin{equation}\label{eq:contour_int}
\int_{a-bi}^{a+1-bi} F(z)dz + \int_{a+1-bi}^{a+1+bi} F(z)dz + \int_{a+1+bi}^{a+bi} F(z)dz + \int_{a+bi}^{a-bi} F(z)dz = 0.
\end{equation}
Now, to estimate the third integral, a change of variable yields
\begin{eqnarray*}
\int_{a+1+bi}^{a+bi} F(z)dz =\int_{a+1+bi}^{a+bi} p_{-\alpha z}(t)q^{-z}\Ehat^*(z)dz  &=& - \int_a^{a+1} t^{-\alpha (y + bi)} q^{-(y+bi)} \Ehat^*(y+bi)dy.
\end{eqnarray*}
Thus, using ~\eqref{eq:f-hat-stirling}, we obtain
\begin{eqnarray*}
\left| \int_{a+1+bi}^{a+bi} F(z)dz \right| \le t^{-\alpha a} q^{-a} C_a b^{a-\frac{1}{2}} e^{-b \frac{\pi}{2}}
\end{eqnarray*}
and therefore $ \int_{a+1+i \infty}^{a+i \infty} F(z)dz = 0$. Similarly, one can show that $\int_{a-i \infty}^{a+1-i \infty} F(z)dz = 0$. Hence, we deduce from \eqref{eq:contour_int} that
\begin{equation*}
\int_{a+1- i \infty}^{a+1 + i \infty} F(z)dz = \int_{a- i \infty}^{a + i \infty} F(z)dz
\end{equation*}
which completes the proof of the identity \eqref{eq:d_qE}. Finally, the additional condition of the second part of item~\eqref{item4}, that is $r \mapsto m_{\alpha}(r)$ is a non-decreasing function on $(0,1)$, yields that  the mapping $y \mapsto m_{\alpha}(e^{-y})$ defined on $\R_+$ is the tail of a L\'evy measure of a subordinator. Thus, it follows from \cite[Proposition 2.1]{patie2012extended} that  $\bPhi $ is a Bernstein function. Furthermore,  easy algebra yields that $\pmb{\phi}'(0^+)=-\frac{\phi(-\alpha)}{\alpha}$ which is finite if and only if $\mathfrak{a}_{\phi}<-\alpha$, and this concludes the proof of item \eqref{item4}. 

Finally, to prove item~\eqref{item3}, making a change of variables and performing an integration by parts in~\eqref{eq:boldA}, we have
\begin{eqnarray*}
\mathbf{A}f(t) &=& t^{-\alpha}\left(btf'(t)-\int_0^{\infty} (f(te^y)-f(t)) d m(e^{-y})\right) \\
&=&  t^{-\alpha}\left(btf'(t) + \int_0^{\infty} te^{y} f'(te^{y})m(e^{-y})dy \right) \\
&=&  t^{-\alpha}\left(btf'(t) + \int_t^{\infty} f'(r)m \left( \frac{t}{r} \right) dr \right).
\end{eqnarray*}
Then, recalling that $\Lambda f = f\circ \iota$ with $\iota(y)=\frac{1}{y}$, and making another  change of variable, we obtain that
\begin{eqnarray*}
\mathbf{A}\Lambda f(t) &=& t^{-\alpha}\left(bt\frac{-1}{t^2}f'\left(\frac{1}{t} \right) + \int_t^{\infty} \frac{-1}{r^2}f'\left(\frac{1}{r} \right) m \left( \frac{t}{r} \right) dr \right)\\
&=& - t^{-\alpha}\left(b\frac{1}{t}f'\left(\frac{1}{t} \right) + \int_0^{\frac{1}{t}} f'(y) m \left( \frac{y}{1/t} \right) dy \right)
\end{eqnarray*}
and thus
\begin{eqnarray*}
  \Lambda \mathbf{A}\Lambda f(t) = - t^{\alpha}\left(bt f'\left (t\right) + \int_0^{t} f'(y) m \left( \frac{y}{t} \right) dy \right)= -t^{2\alpha}\opphi f (t)
\end{eqnarray*}
from where we conclude the proof of the intertwining relation by using the fact that $\Lambda$ is an involution.
\end{proof}
\section{Self-similar Cauchy problem and stochastic representation}\label{sec:frac_Cauchy}
Let $X=(X_t)_{t \geq 0}$ be a strong Markov process defined on a filtered probability space $(\Omega, \F, (\F_t)_{t\geq 0}, \Prob )$ and taking values in  $E \subset \R^d, d \in \N$,   endowed with a sigma-algebra $\mathcal{E}$. We denote its associated semigroup by $P=(P_t)_{t \geq 0}$ which is defined, for any $t \geq 0$ and $f$ a bounded Borelian function, by
\[ P_t f(x) = \E_x[f(X_t)]\]
where $\E_x$ stands for the expectation operator with respect to $\Prob_x(X_0 = x) =1$. Since $x \mapsto \E_x$ is $\mathcal{E}$-measurable, for any Radon measure $\nu$, we use the notation
\[ \nu P_t f = \E_{\nu}[f(X_t)] = \int_E \E_x[f(X_t)]\nu(dx).\]
We say that a Radon measure $\nu$ is an \emph{invariant measure} if for all $t \geq 0$, $\nu P_t f = \nu f$. Now, since $\nu$ is non-negative on $E$, we define the weighted Hilbert space
\begin{equation*}
L^2(\nu) = \{ f :E\rightarrow \R \textrm{ measurable};\ \int_E f^2(x)\nu(dx)<\infty \}
\end{equation*}
endowed with the inner product $\langle \cdot, \cdot \rangle_{\nu}$, where $\langle f, g \rangle_{\nu} = \int_E f(x)g(x)\nu(dx)$, and norm $\Vert f  \Vert_{\nu} = \sqrt{\langle f,f \rangle_{\nu}}$.  We simply write $L^2(\R_+)$ when $\nu$ is the Lebesgue measure on $\R_+$. Then, a classical result yields that we can extend $P$ as a strongly continuous contraction Markov semigroup in $L^2(\nu)$, and when there is no confusion, we still denote this extension by $P$.
We denote by $(\bL,\calD(\bL))$ the infinitesimal generator of the semigroup $P$, i.e.
\[\calD(\bL) = \lbrace f \in L^2(\nu); \bL f = \lim_{t \rightarrow 0}\frac{P_t f - f}{t} \in L^2(\nu) \rbrace.\]
In order to provide a stochastic and explicit representation of the solution to the self-similar Cauchy problem, we shall consider two different cases, for which we recall that as bounded family of operators $P$ admits an adjoint semigroup $P^*=(P^*_t)_{t\geq0},$ which is defined, for all $t\geq0$, by $\langle P_tf, g \rangle_{\nu}=\langle f, P^*_tg \rangle_{\nu}$. We say that $P$ is normal (resp.~self-adjoint) if $P_tP^*_t=P^*_tP_t$ (resp.~$P_t=P^*_t$), and of course the second property is stronger.
\begin{assumption}\label{assump-L}
 $P$ is a normal semigroup on $L^2(\nu)$.
\end{assumption}
Note that under Assumption~\ref{assump-L}, $\bL$ is a non-negative, densely defined and normal operator on $L^2(\nu)$, and there is a unique resolution $\II$ of the identity, supported on $\sigma(\bL)$,  the spectrum of $\bL$,  where for any $\lambda \in \sigma(\bL)$, $\Re(\lambda)\geq0$,
\begin{equation}\label{eq:resolution}
\bL = \int_{\sigma(\bL)}-\lambda d\II(\lambda)
\end{equation}
with the domain~$\calD(\bL) =\lbrace f \in L^2(\nu); \int_{\sigma(\bL)}|\lambda|^2 d\II_{f,f}(\lambda)<\infty \rbrace$, see e.g.~\cite[Chapter IX]{Rudin1973}. The identity~\eqref{eq:resolution} is a shorthand notation that means
\begin{equation*}
\langle \bL f, g \rangle_{\nu} = \int_{\sigma(\bL)}-\lambda d\II_{f,g}(\lambda), \quad f\in \calD(\bL), \: g\in L^2(\nu),
\end{equation*}
where $d \II_{f,g}(\lambda)$ is a regular Borel complex measure of bounded variation concentrated on $\sigma(\bL)$, with $d |\II_{f,g}|({\sigma(\bL)}) \le \Vert f \Vert_{\nu} \Vert g \Vert_{\nu}$. Then, for $\psi$ a real measurable function defined on $\sigma(\bL)$, the operator $\psi(\bL)$ is given  by
\begin{equation*}
\psi(\bL) = \int_{\sigma(\bL)}\psi(-\lambda)d\II(\lambda) \text{ with the domain } \calD(\psi(\bL)) =\lbrace f \in L^2(\nu); \int_{\sigma(\bL)}|\psi(-\lambda)|^2 d\II_{f,f}(\lambda)<\infty \rbrace.
\end{equation*}
We point out  that  spectral theoretical arguments have already been used in the context of the fractional Cauchy problems associated to self-adjoint operators, see e.g.~\cite{chen2012space},~\cite{meerschaert2009fractional},~\cite{meerschaert2009distributed},~\cite{meerschaert_toaldo}.

Next, we say that sequences $(\Poly_n)_{n \geq 0}$ and $(\Nu_n)_{n \geq 0} $ are biorthogonal in $L^2(\nu)$ if they both belong to $ L^2(\nu)$ and $\langle \Poly_m, \Nu_n \rangle_{\nu} = \mathbf{I}_{\{m=n\}}$. Moreover, a sequence that admits a biorthogonal sequence will be called minimal and a sequence that is both minimal and complete, in the sense that its linear span is dense in $L^2(\nu)$, will be called \emph{exact}. It is easy to show that a sequence $(\Poly_n)_{n \geq 0}$ is minimal if and only if none of its elements can be approximated by linear combinations of the others. Next, recall that $(\Poly_n)_{n \geq 0}$ form a Bessel sequence in $L^2(\nu)$ with bound $B>0$, if for any $f \in L^2(\nu)$,
\begin{equation}\label{eq:bes}
\sum_{n=0}^{\infty}|\langle f, \Poly_n \rangle_{\nu}|^2 \le B \Vert f \Vert^2_{\nu}.
\end{equation}
Then, the so-called synthesis operator $S:l^2(\N) \rightarrow L^2(\nu)$ defined by
$$S:\underline{c} = (c_n)_{n \geq 0} \mapsto S(\underline{c})= \sum_{n=0}^{\infty} c_n \Poly_n$$
is a bounded operator with norm $\Vert S \Vert_{\nu} \le \sqrt{B}$, i.e. the series is norm-convergent for any sequence $(c_n)_{n \geq 0}$ in $l^2(\N)$. Furthermore, when $(\Poly_n)_{n \geq 0}$ is an orthogonal system, in ~\eqref{eq:bes} we also have a lower bound and the operator $S$ is invertible.

\begin{assumption}\label{assump-P}
Assume that $P$ admits the following spectral expansion, for any $f \in \mathcal{D}$ with $\overline{\mathcal{D}}=L^2(\nu)$, and $t > T$ for some $T > 0$,
\begin{equation}\label{eq:spect_exp}
P_t f = \sum_{n=0}^{\infty}e^{-\lambda_n t} \langle f, \Nu_n \rangle_{\nu} \Poly_n \quad \text{in } L^2(\nu)
\end{equation}
where $(\lambda_n)_{n\geq 0} \in \mathbb{C}$,  with $\Re(\lambda_n)\geq 0, n \geq 0$, is the sequence of the ordered (in modulus) eigenvalues associated to the sequence of eigenfunctions $(\Poly_n)_{n \geq 0}$ which is an exact Bessel sequence in $L^2(\nu)$ with Bessel bound $B>0$, and $(\Poly_n,\Nu_n)_{n \geq 0}$ form a biorthogonal sequence in $L^2(\nu)$.
\end{assumption}
Note that when $P$ is self-adjoint, then $\Poly_n = \Nu_n$, $\forall n \in \N$, and $(\Poly_n)_{n \geq 0}$ form an orthogonal basis of $L^2(\nu)$ and ~\eqref{eq:spect_exp} is valid for all $t \geq 0$. In general, $(\Poly_n,\Nu_n)_{n \geq 0}$ do not need to form a basis.
Now, let $\zeta$ be the right-inverse of the  non-decreasing  $\alpha$-self-similar Markov process, with $0<\alpha<1$, associated via the Lamperti's mapping with $\phi$ defined by $\eqref{LKs}$.
Recall that if $\phi \in \BD$, then Proposition~\ref{prop:L_moments} implies that for $q,t > 0$,
\begin{equation}\label{eq:Laplace-transform}
\E\left[e^{-q \zeta_t}\right]=\int_0^{\infty} e^{-qs} \Prob(\zeta_t \in ds) =\Ep(-q t^{\alpha})
\end{equation}
which either admits the series or the Mellin-Barnes integral representation provided in Proposition~\ref{prop:L_moments}.
We denote the time-changed process by $X_{\zeta}=(X_{\zeta_t})_{t \geq 0}$, and for $f \in L^2(\nu)$, define the family of linear operators $P^{\phi_{\alpha}}=(P^{\phi_{\alpha}}_t)_{t \geq 0}$ by the Bochner integral
\begin{equation}\label{eq:Bochner_integral}
P^{\phi_{\alpha}}_t f(x) =\E_x[f(X_{\zeta_t})] = \int_0^{\infty}P_s f(x) \Prob(\zeta_t \in ds).
\end{equation}
Throughout this section we assume that $\phi \in \BD $, and recall that $\bPhi(u)=\frac{u}{u-\alpha}\phi(u-\alpha)$, $u>0$, is well-defined. Then, we define the set of functions
\begin{equation*}
\DL = \left\lbrace f \in L^2(\nu);\ \left( \lambda_n \langle f, \Nu_n \rangle_{\nu} \right)_{n \geq 0} \in l^2(\N) \right\rbrace \subseteq \calD(\bL)
\end{equation*}
and
since clearly $Span(\Poly_n) \subseteq \DL$  and by Assumption~\ref{assump-P}, ${Span}(\Poly_n) $ is dense in $ L^2(\nu)$,  we have $\DL$ is also dense in $L^2(\nu)$. We are now ready to state  the last main result of this paper.
\begin{theorem}\label{thm:Cauchy}
Let $\phi \in \BD $.
If Assumption~\ref{assump-L} (resp.~Assumption~\ref{assump-P}) holds, then for any $f \in \calD(\bL)$ (resp.~$f \in \DL$), the function $u(t,x)=P^{\phi_{\alpha}}_tf(x)$, is a strong solution in $L^2(\nu)$ to
\begin{eqnarray*}
\opbPhi u(t,x) &=& \mathbf{L} u(t,x), \quad  t>0  \: (resp.~t >T),\\
u(0,x) &=& f(x)
\end{eqnarray*}
in the following sense:
$t \mapsto u(t,\cdot) \in C^1_0((0,\infty), L^2(\nu))$ (resp.~$ C^1_0((T,\infty), L^2(\nu))$), and both $t \mapsto u(t,\cdot)$ and $t\mapsto \mathbf{L}u(t,\cdot)$ are analytical on the half plane  $\mathfrak{R}(z)>0$ (resp.~$\mathfrak{R}(z)>T$).
Moreover, if Assumption~\ref{assump-L} holds, then for any $f\in \calD(\bL)$ and $t > 0$, $P^{\phi_{\alpha}}_t f$ admits the following spectral representation
\begin{equation}
P^{\phi_{\alpha}}_t f = \int_{\sigma(\bL)}\Ep(-\lambda t^{\alpha}) d\II(\lambda)f \quad \text{in } L^2(\nu).
\end{equation}
Otherwise if Assumption~\ref{assump-P} holds, then, for any $f \in \DL$ and $t > T$,
\begin{equation}\label{eq:P_eta}
P^{\phi_{\alpha}}_t f = \sum_{n=0}^{\infty} \Ep(-\lambda_n t^{\alpha})\langle f, \Nu_n \rangle_{\nu} \Poly_n \quad \text{in } L^2(\nu).
\end{equation}
\end{theorem}
%
\begin{proof}
First, note that since  $P_t$ is for all $ t \geq 0$ a contraction, using Bochner's inequality, see \cite[Theorem 1.1.4]{book_ABHN}, one can note from~\eqref{eq:Bochner_integral} that, for any $f \in L^2(\nu)$,
\begin{equation*}
\Vert P^{\phi_{\alpha}}_t f \Vert_{\nu} = \left\Vert \int_0^{\infty} P_s f \Prob(\zeta_t \in ds) \right\Vert_{\nu} \le \int_0^{\infty} \Vert P_s f \Vert_{\nu} \Prob(\zeta_t \in ds) \le \Vert f \Vert_{\nu}.
\end{equation*}
Thus, for any $t \geq 0$, $P^{\phi_{\alpha}}_t$ is a bounded operator in $L^2(\nu)$.
Now, let Assumption~\ref{assump-L} holds. Then, by the functional calculus, we have that for all $t>0$
\begin{equation*}
P_t = e^{t \bL} = \int_{\sigma(\bL)} e^{-t \lambda} d\II(\lambda).
\end{equation*}
Therefore, $\zeta$ being the right-inverse of the non-decreasing self-similar Markov process associated to $\phi \in \BD $, we have, using the identity  \eqref{eq:Bochner_integral}, that for any $f \in L^2(\nu)$ and $t>0$,
\begin{eqnarray}
P^{\phi_{\alpha}}_t f = \int_0^{\infty} P_s f\Prob(\zeta_t \in ds) &=& \int_0^{\infty}\int_{\sigma(\bL)} e^{-s \lambda} d\II(\lambda)f\Prob(\zeta_t \in ds) \nonumber \\
&=& \int_{\sigma(\bL)}\int_0^{\infty} e^{-s \lambda}\Prob(\zeta_t \in ds) d\II(\lambda)f \nonumber \\
&=& \int_{\sigma(\bL)}\Ep(-\lambda t^{\alpha}) d\II(\lambda)f \label{eq:I-int}
\end{eqnarray}
where for the transition from second to third equality, we used Fubini's theorem under the inner product $\langle \cdot, \cdot \rangle_{\nu} = \Vert \cdot \Vert_{\nu}^2$, by a simple polarization argument, which is allowed since the measure $d\II$ is of bounded variation on $\sigma(\bL)$ and, as a Laplace transform of a probability measure, for all $t,\Re(\lambda)\geq 0$, $|\Ep(-\lambda t^{\alpha})|\le 1$, and  for the last step  we used the identity \eqref{eq:Laplace-transform}. Now, as  for any $t \geq 0$,  $P_t^{\phi_{\alpha}}$ is  bounded in $L^2(\nu)$, we have $P^{\phi_{\alpha}}_t \bL \subseteq \bL P^{\phi_{\alpha}}_t$ and thus $\calD(P^{\phi_{\alpha}}_t \bL)=\calD(\bL) \subseteq \calD(\bL P^{\phi_{\alpha}}_t) = \{ f\in L^2(\nu); P^{\phi_{\alpha}}_t f \in \calD(\bL) \}$, see~\cite[Theorem 13.24, (15) and (10)]{Rudin1973}. Hence, we conclude that $P^{\phi_{\alpha}}_t$ maps $\calD(\bL)$ into itself, and since $P^{\phi_{\alpha}}_t f \in \calD(\bL)$ for all $f \in \calD(\bL)$, by the functional calculus, we obtain
\begin{equation}
\bL P^{\phi_{\alpha}}_t f = \int_{\sigma(\bL)}-\lambda \Ep(-\lambda t^{\alpha}) d\II(\lambda)f.
\end{equation}
Next, since by Proposition~\ref{prop:L_moments}\eqref{it3}, $t \mapsto \Ep(-t) \in C^{\infty}_0(\R_+)$, then, for $\Re(\lambda)>0$, the mapping $t \mapsto \Ep(-\lambda t^{\alpha}) \in C^{\infty}_0(\R_+)$ and
\begin{equation}\label{eq:der_Ep}
 \frac{d}{dt} \Ep(-\lambda t^{\alpha}) = \frac{d}{dt} \E[e^{-\lambda t^{\alpha} \zeta_1}] = -\lambda \alpha t^{\alpha-1} \E[\zeta_1 e^{-\lambda t^{\alpha} \zeta_1}],
\end{equation}
which is bounded on $t \in [t_0,\infty)$ for any $t_0>0$ and $\Re(\lambda)\geq 0$ since by item~\eqref{it1} of Proposition~\ref{prop:L_moments}, $\E[\zeta_1] = \frac{1}{\phi_{\alpha}'(0^+)}<\infty$. Furthermore, since we have, for any $t,s>0$,
\begin{equation*}
\Vert P^{\phi_{\alpha}}_t f- P^{\phi_{\alpha}}_s f\Vert_{\nu}^2 = \int_{\sigma(\bL)}(\Ep(-\lambda t^{\alpha}) - \Ep(-\lambda s^{\alpha}))^2 d\II_{f,f}(\lambda)
\end{equation*}
and
\begin{eqnarray*}
&&\left\Vert \frac{P^{\phi_{\alpha}}_t f- P^{\phi_{\alpha}}_s f}{t-s} - \int_{\sigma(\bL)} \frac{d}{dt} \Ep(-\lambda t^{\alpha}) d\II(\lambda)f \right\Vert_{\nu}^2 \\ &=& \int_{\sigma(\bL)} \left( \frac{\Ep(-\lambda t^{\alpha}) - \Ep(-\lambda s^{\alpha})}{t-s} -  \frac{d}{dt} \Ep(-\lambda t^{\alpha}) \right)^2 d\II_{f,f}(\lambda)
\end{eqnarray*}
we obtain that, in the Hilbert space topology, $t \mapsto P^{\phi_{\alpha}}_t$ is also continuously differentiable vanishing along with its derivative at $\infty$, i.e.~it is in $C^1_0((0,\infty), L^2(\nu))$. Indeed, the last identity entails that for any $t>0$,
\begin{equation}\label{eq:der_Pt_phi}
\frac{d}{dt} P^{\phi_{\alpha}}_t = \int_{\sigma(\bL)} \frac{d}{dt} \Ep(-\lambda t^{\alpha}) d\II(\lambda)
\end{equation}
where we note that for any $t>0$ and $f \in \calD(\bL)$,
\begin{eqnarray}
\left\Vert \int_{\sigma(\bL)} \frac{d}{dt} \Ep(-\lambda t^{\alpha}) d\II(\lambda) f \right\Vert_{\nu}^2 &=& \int_{\sigma(\bL)}\left(  \frac{d}{dt} \Ep(-\lambda t^{\alpha}) \right)^2 d\II_{f,f}(\lambda) \nonumber \\
& \le & \left(\frac{\alpha t^{\alpha-1}}{\phi'_{\alpha}(0^+)} \right)^2 \int_{\sigma(\bL)} |\lambda |^2 d\II_{f,f}(\lambda) < \infty \label{eq:der_Ep_norm}
\end{eqnarray}
where we used ~\eqref{eq:der_Ep}, and once again  that $|\E[\zeta_1 e^{-\lambda t^{\alpha} \zeta_1}]| \le \E[\zeta_1] = \frac{1}{\phi'_\alpha(0^+)}$ for any $t, \Re(\lambda)\geq 0$.
Then, by ~\eqref{eq:der_Pt_phi} and Proposition~\ref{prop:operator}\eqref{item4}, we have that
\begin{eqnarray*}
\opbPhi P^{\phi_{\alpha}}_t f &=& t^{1-\alpha}b \frac{d}{dt} P^{\phi_{\alpha}}_t f + t^{-\alpha} \int_0^t \frac{d}{dt} P^{\phi_{\alpha}}_t f\ m\left(\frac{y}{t} \right)dy \\
&=& t^{1-\alpha}b \int_{\sigma(\bL)} \frac{d}{dt} \Ep(-\lambda t^{\alpha}) d\II(\lambda)f + t^{-\alpha} \int_0^t \int_{\sigma(\bL)} \frac{d}{dy} \Ep(-\lambda y^{\alpha}) d\II(\lambda)f\ m\left(\frac{y}{t} \right)dy \\
&=&  \int_{\sigma(\bL)} \opbPhi \Ep(-\lambda t^{\alpha}) d\II(\lambda)f = \int_{\sigma(\bL)} -\lambda \Ep(-\lambda t^{\alpha}) d\II(\lambda)f
\end{eqnarray*}
where in the second step to change the order of integration, we used Fubini's theorem for Bochner integrals, see \cite[Theorem 1.1.9]{book_ABHN}, which is justified since by~\eqref{eq:der_Ep_norm} we have
\begin{eqnarray*}
 \int_0^t \left\Vert \int_{\sigma(\bL)} \frac{d}{dy} \Ep(-\lambda y^{\alpha}) d\II(\lambda)f \right\Vert_{\nu}  m\left(\frac{y}{t}\right)dy &\le & \frac{\alpha}{\phi_{\alpha}'(0^+)} \left(\int_{\sigma(\bL)} |\lambda |^2 d\II_{f,f}(\lambda) \right)^{\frac{1}{2}} \int_0^t y^{\alpha-1}m\left(\frac{y}{t}\right)dy \\
 &\le & \frac{\alpha t^{\alpha}}{\phi_{\alpha}'(0^+)} \left(\int_{\sigma(\bL)} |\lambda |^2 d\II_{f,f}(\lambda) \right)^{\frac{1}{2}} \int_0^1 r^{\alpha-1} m(r) dr <\infty
\end{eqnarray*}
since $\alpha \in (0,1)$.
Thus, $\bL P^{\phi_{\alpha}}_t f = \opbPhi P^{\phi_{\alpha}}_t f $, and taking $t=0$ in \eqref{eq:Bochner_integral}, we easily check that $u(0,x)= f(x)$, $x \in E$.
Now, let us assume that Assumption~\ref{assump-P} holds, and define the family of linear operators $S^{\phi_{\alpha}}=(S^{\phi_{\alpha}}_t)_{t > T}$, for
$f \in \DL$ and $t > T$, by
\begin{equation}\label{eq:S_eta}
S^{\phi_{\alpha}}_t f = \sum_{n=0}^{\infty} \int_{0}^{\infty}  {e^{-\lambda_n s}\ \Prob(\zeta_t \in ds)} \langle f, \Nu_n \rangle_{\nu} \Poly_n
= \sum_{n=0}^{\infty} \Ep(-\lambda_n t^{\alpha}) \langle f, \Nu_n \rangle_{\nu} \Poly_n.
\end{equation}
Note that $S^{\phi_{\alpha}}_t f \in L^2(\nu)$ for any $f \in \DL$. Indeed, recalling that $\Re(\lambda_n) \geq 0$, $n = 0, 1, \cdots$,
as a Laplace transform of a probability measure, $|\Ep(-\lambda_n t^{\alpha})| \le 1$ for any $t \geq 0$ and $n =0,1,\cdots$, we have
\begin{eqnarray*}
\sum_{n=0}^{\infty} |\mathcal{E}_{\phi_{\alpha}}(-\lambda_n t^{\alpha})|^2 |\langle f, \Nu_n \rangle_{\nu}|^2 \le \sum_{n=0}^{\infty}  |\langle f, \Nu_n \rangle_{\nu}|^2 \le M + \sum_{n=\m}^{\infty} |\lambda_n|^2 |\langle f, \Nu_n \rangle_{\nu}|^2  <\infty
\end{eqnarray*}
where $\m = \min \{k \geq 0;\ |\lambda_k| \geq 1\}$ and in which case there exists $M\geq 0$ such that $\sum_{n=0}^{\m-1}|\langle f, \Nu_n \rangle_{\nu}|^2 \le M $.
Moreover, by the Bessel property of $(\Poly_n)_{n \geq 0}$, we have that $S^{\phi_{\alpha}}_t$ is a  bounded operator on $\DL$ with $\Vert S^{\phi_{\alpha}}_t \Vert_{\nu} \le \sqrt{B}$.
Furthermore, since $\langle \Poly_m, \Nu_n \rangle_{\nu} = \mathbf{I}_{\{m=n\}}$, we have, for any $m \in \N$,
\begin{equation*}
S^{\phi_{\alpha}}_t \Poly_m = \sum_{n=0}^{\infty} \Ep(-\lambda_n t^{\alpha}) \langle \Poly_m, \Nu_n \rangle_{\nu} \Poly_n = \Ep(-\lambda_m t^{\alpha}) \Poly_m.
\end{equation*}
On the other hand, recalling the spectral expansion of $P_t$ given in~\eqref{eq:spect_exp}, we have, for $t > T$,
\begin{equation*}
P^{\phi_{\alpha}}_t \Poly_m = \int_{0}^{\infty} \sum_{n=0}^{\infty} e^{-\lambda_n s} \langle \Poly_m, \Nu_n \rangle_{\nu} \Poly_n\ \Prob(\zeta_t \in ds) = \int_{0}^{\infty} e^{-\lambda_m s} \Poly_m\ \Prob(\zeta_t \in ds) = \Ep(-\lambda_m t^{\alpha}) \Poly_m.
\end{equation*}
Thus, $P^{\phi_{\alpha}}_t$ and $S^{\phi_{\alpha}}_t$ coincide on ${Span}(\Poly_n)$, and since $\overline{Span}(\Poly_n)= L^2(\nu) \supseteq \DL$, the bounded linear transformation Theorem implies that $P^{\phi_{\alpha}}_t = S^{\phi_{\alpha}}_t$ on $\DL$ when $t > T$.
Next, since for all $n$, $\Poly_n $ is an  eigenfunction, $\Poly_n \in L^2(\nu)$, $P_t \Poly_n = e^{-\lambda_n t}\Poly_n$ and hence $\Poly_n \in \mathcal{D}(\mathbf{L})$ with $\mathbf{L}\Poly_n = - \lambda_n \Poly_n$.
Thus, by linearity, for any $t\geq 0$ and $N = 1,2, \cdots$, $h_t^N\in \mathcal{D}(\mathbf{L})$, where $h_t^N = \sum_{n=0}^{N} \Ep(-\lambda_n t^{\alpha}) \langle f, \Nu_n \rangle_{\nu}\Poly_n$, $f\in \calD(\bL)$, and
\begin{equation*}
\mathbf{L}h_t^N = \sum_{n=0}^{N} \Ep(-\lambda_n t^{\alpha}) \langle f, \Nu_n \rangle_{\nu}\mathbf{L} \Poly_n = \sum_{n=0}^{N} -\lambda_n \Ep(-\lambda_n t^{\alpha}) \langle f, \Nu_n \rangle_{\nu} \Poly_n.
\end{equation*}
Then, letting $N \rightarrow \infty$, we obtain
\begin{eqnarray*}
h_t^N &=& \sum_{n=0}^{N} \Ep(-\lambda_n t^{\alpha}) \langle f, \Nu_n \rangle_{\nu}\Poly_n  \rightarrow P^{\phi_{\alpha}}_tf, \quad \text{and}\\
\mathbf{L} h_t^N &=& \sum_{n=0}^{N} -\lambda_n \Ep(-\lambda_n t^{\alpha}) \langle f, \Nu_n \rangle_{\nu} \Poly_n \rightarrow \sum_{n=0}^{\infty} -\lambda_n \Ep(-\lambda_n t^{\alpha}) \langle f, \Nu_n \rangle_{\nu} \Poly_n.
\end{eqnarray*}
Observing that, since $|\Ep(-\lambda_n t^{\alpha})| \le 1$, for any $n =0,1,2,\cdots$, $t \geq 0$ and $f \in \DL \subseteq  \calD(\bL)$,
\begin{equation*}
\sum_{n=0}^{\infty} |-\lambda_n \Ep(-\lambda_n t^{\alpha}) \langle f, \Nu_n \rangle_{\nu}|^2 \le \sum_{n=0}^{\infty} \lambda_n^2 |\langle f, \Nu_n \rangle_{\nu}|^2 < \infty
\end{equation*}
and thus the Bessel property of $(\Poly_n)_{n\geq 0}$ implies that $\sum_{n=0}^{\infty} -\lambda_n \Ep(-\lambda_n t^{\alpha}) \langle f, \Nu_n \rangle_{\nu} \Poly_n \in L^2(\nu)$. Therefore, since the operator $\mathbf{L}$ is closed, we obtain that $P^{\phi_{\alpha}}_t f \in \mathcal{D}(\mathbf{L})$ and
\begin{equation}\label{eq:L_P_eta}
\mathbf{L} P^{\phi_{\alpha}}_tf = \sum_{n=0}^{\infty}  -\lambda_n \Ep(-\lambda_n t^{\alpha}) \langle f, \Nu_n \rangle_{\nu}\Poly_n.
\end{equation}
Now, similar to the justification under Assumption~\ref{assump-L} above, one can show that for any $f \in \DL$, the mapping $t\mapsto  P^{\phi_{\alpha}}_t$ is a $C^1_0((T,\infty), L^2(\nu))$ function,  and for any $t > T$, ~\eqref{eq:der_Pt_phi} holds. Then,  for any $f \in \DL$ and $t > T$, we have
\begin{eqnarray*}
\opbPhi P^{\phi_{\alpha}}_t f &=& \opbPhi \sum_{n=0}^{\infty} \Ep(-\lambda_n t^{\alpha}) \langle f, \Nu_n \rangle_{\nu} \Poly_n \\
&=& \sum_{n=0}^{\infty} \opbPhi \Ep(-\lambda_n t^{\alpha}) \langle f, \Nu_n \rangle_{\nu} \Poly_n \\
&=& \sum_{n=0}^{\infty} -\lambda_n \Ep(-\lambda_n t^{\alpha}) \langle f, \Nu_n \rangle_{\nu} \Poly_n \in L^2(\nu)
\end{eqnarray*}
where we noted that we are allowed to change the order of the operator $\opbPhi$ and summation similar to the case of the normal operator above. Indeed, to change the order of summation and integration, using the Bessel property of $(\Poly_n)_{n \geq 0}$ and recalling the definition of $\DL$, we apply Fubini's theorem.
Thus, we conclude that for $f \in \DL$ and $t > T$, $\opbPhi P^{\phi_{\alpha}}_t f = \mathbf{L}P^{\phi_{\alpha}}_t f$.
Moreover, taking $t = 0$ in \eqref{eq:Bochner_integral}, one can easily check that $u(0,x) = f(x)$ for $x \in E$.
Finally, under Assumption~\ref{assump-L} (resp.~Assumption~\ref{assump-P}), given the eigenvalues expansion of $P^{\phi_{\alpha}}_t$, we have that $t \mapsto u(t,\cdot)=P^{\phi_{\alpha}}_t f$ and $t \mapsto \mathbf{L}u(t,\cdot)=P^{\phi_{\alpha}}_t\mathbf{L}f$ are analytical on the half plane $\mathfrak{R}(z)>0$ (resp.~$\mathfrak{R}(z)>T$), and this concludes the proof.
\end{proof}

\section{Examples}\label{sec:Examples}
Let $\zeta = (\zeta_t)_{t \geq 0}$ be the inverse of the non-decreasing $\alpha$-self-similar Markov process $\chi = (\chi_t)_{t \geq 0}$ defined in Section~\ref{sec:self-sim}, and associated via the Lamperti mapping to the subordinator with a Laplace exponent $\phi \in \BD $, defined by \eqref{LKs}. Furthermore, recall that $\bPhi$ is defined by \eqref{eq:bPhi}. In this section, we consider some examples that illustrate the variety of applications of our main results and they cover the both situations when  Assumption~\ref{assump-L} or Assumption~\ref{assump-P} holds. Namely, section~\ref{sec:diffusions} and~\ref{sec:non-diffusions} include examples of self-adjoint, and non-self-adjoint and non-local semigroups respectively.

\subsection{Some self-adjoint examples}\label{sec:diffusions}
\subsubsection{Squared Bessel semigroups}
We consider first the case where $P=(P_t)_{t\geq 0}$ is the semigroup of the squared Bessel process of order $2$, that is  its infinitesimal generator is given, for a smooth function $f$, by
\begin{equation}\label{eq:L_Q}
\bL f(x)=2 x f''(x)+2 f'(x), \quad x>0.
\end{equation}
It is well known that $P_t$ is a strongly continuous contraction semigroup and self-adjoint  in $L^2(\R_+)$.
Next, we define the function $J$, for $z\in\mathbb{C}$, by
\begin{equation*}
J(z)= \sum_{n=0}^{\infty}\frac{(e^{i\pi}z)^n}{(n!)^2}
\end{equation*}
and observe that $J\left(\frac{z^2}{4}\right)=\mathrm{J}_0(z)$, where $\mathrm{J}_0$ is the Bessel function of the first kind of order $0$. We also recall that $H$ the Hankel transform associated to $J$ is  an involution of $L^2(\R_+)$, i.e. $HH$ is the identity, defined  by
\begin{equation*}
H f(x)=\int_0^{\infty} J(\lambda x)f(\lambda )d\lambda.
\end{equation*}
Then, $P$ admits the following spectral expansion, for any $t>0$ and $f \in L^2(\R_+)$,
\begin{eqnarray*}
P_t f =  H \e_{t} H f,
\end{eqnarray*}
where we set $\e_{t}(x)=e^{-tx}$, see e.g.~\cite{reflected_stable}. Then, since Assumption~\ref{assump-L} is satisfied, Theorem~\ref{thm:Cauchy} implies that, for any $f \in \calD(\bL)$, $P^{\phi_{\alpha}}_t f$ solves the self-similar Cauchy problem,
\begin{eqnarray*}
\opbPhi u(t,x) &=& \bL u(t,x), \quad  t >0,\\
u(0,x) &=& f(x).
\end{eqnarray*}
Furthermore, the solution has the following spectral representation, for all $t>0$,
\begin{equation}
P^{\phi_{\alpha}}_t f(x) = \int_0^{\infty} \Ep(-\lambda t^{\alpha}) H f(\lambda) J(\lambda x)d\lambda \quad \text{in } L^2(\R_+).
\end{equation}

\subsubsection{The classical Laguerre semigroup}
Let $P=(P_t)_{t \geq 0}$ be the classical Laguerre semigroup of order $0$, i.e.~its infinitesimal generator takes the form, for a smooth function $f$,
\begin{equation*}
\bL f(x) = x f''(x) + (1-x) f'(x), \quad x>0,
\end{equation*}
see e.g.~\cite[Section 3.1]{patie2015spectral}. Then, the semigroup $P$ is a self-adjoint and strongly continuous contraction semigroup on the weighted Hilbert space $L^2(\nu)$ with $\nu(dx)=e^{-x}dx$, $x>0$, which is the  unique invariant measure.  Moreover, it admits the eigenvalues expansions, valid for any $t>0$,
\begin{equation*}
P_t f = \sum_{n=0}^{\infty}e^{-nt} \langle f, \calL_n \rangle_{\nu}\calL_n  \quad \text{in } L^2(\nu)
\end{equation*}
where for any $n \geq 0$, $\calL_n$ is the Laguerre polynomial of order $0$, defined through the polynomial representation
\begin{equation*}
\calL_n (x) = \sum_{k=0}^n (-1)^k \binom{n}{k}\frac{x^k}{k!}.
\end{equation*}
Since $P$ is self-adjoint in $L^2(\nu)$,  Assumption~\ref{assump-L} is satisfied with $\sigma(\bL)=\{\lambda_n = n, n \geq 0\}$, and it follows from Theorem~\ref{thm:Cauchy} that for any $f \in \calD(\bL)$, $P^{\phi_{\alpha}}_t f$ solves the self-similar Cauchy problem,
\begin{eqnarray*}
\opbPhi u(t,x) &=& \bL u(t,x), \quad  t >0,\\
u(0,x) &=& f(x).
\end{eqnarray*}
Furthermore, the solution has the following spectral representation, for all $t>0$,
\begin{equation*}
P^{\phi_{\alpha}}_t f = \sum_{n=0}^{\infty}\Ep(-nt^{\alpha}) \langle f, \calL_n \rangle_{\nu}\calL_n \quad \text{in } L^2(\nu).
\end{equation*}

\subsubsection{Classical Jacobi semigroups}
Now, assume $\lam>\mu>0$ and let us consider the classical Jacobi semigroup $P=(P_t)_{t \geq 0}$ on $E=(0,1)$, which is a Feller semigroup and its infinitesimal generator $\bL_{\mu}$ has, for any $f \in C^2(E)$, the following form
\begin{equation*}
\bL_{\mu} f(x) = x(1-x) f''(x) - (\lam x - \mu) f'(x),\quad x \in (0,1),
\end{equation*}
see e.g.~\cite[Section 5]{patie2018jacobi}. The classical Jacobi semigroup $P$ admits a unique invariant measure $\beta_{\mu}$, which is the distribution of a beta random variable of parameters $\mu>0$ and $\lam-\mu>0$, i.e.
\begin{equation}\label{eq:beta-dist}
\beta_{\mu}(dy) =\beta_{\mu}(y)dy = \frac{\Gamma(\lam)}{\Gamma(\mu)\Gamma(\lam-\mu)}y^{\mu-1}(1-y)^{\lam-\mu-1}dy, \quad y \in (0,1).
\end{equation}
 Moreover, $P$ extends to a strongly continuous contraction semigroup on $L^2(\beta_{\mu})$ which we still denote by $P$. The eigenfunctions of $P$ are the Jacobi polynomials which form an orthonormal basis in $L^2(\beta_{\mu})$ and are given, for any $n \in \N$ and $x \in E$, by
\begin{equation}\label{J_pol}
\Poly_n^{\lam, \mu} (x) =  \sqrt{C_n(\mu)} \sum_{k=0}^n \frac{(-1)^{n+k}}{(n-k)!}  \frac{(\lam-1)_{n+k}}{(\lam-1)_{n\phantom{+k}}} \frac{(\mu)_n}{(\mu)_k} \frac{x^k}{k!},
\end{equation}
where we have set
\begin{equation}\label{eq:Cn-lam-mo}
C_n(\mu)=(2n+\lam-1)\frac{n!(\lam)_{n-1}}{(\mu)_n(\lam-\mu)_n}.
\end{equation}
Next, the eigenvalue associated to the eigenfunction $\Poly_n$ is, for $n \in \N$,
\begin{equation}\label{Q_eigenv}
\lambda_n =  n^2+(\lam-1)n= n(n-1)+\lam n.
\end{equation}
The semigroup $P$ then admits the spectral decomposition given, for any $f\in L^2(\beta_{\mu})$ and $t \geq 0$, by
\begin{equation}\label{Q_sem}
P_t f = \sum_{n=0}^{\infty} e^{- \lambda_n t} \langle f,\Poly_n^{\lam, \mu} \rangle_{\beta_{\mu}} \Poly_n^{\lam, \mu}.
\end{equation}
Since $P$ is self-adjoint, Assumption~\ref{assump-L} is satisfied with $\sigma(\bL)= \{\lambda_n = n(n-1)+\lam n$, $n \geq 0\}$, and it follows from Theorem~\ref{thm:Cauchy} that for any $f \in \calD(\bL_{\mu})$, $P^{\phi_{\alpha}}_t f$ solves the self-similar Cauchy problem,
\begin{eqnarray*}
\opbPhi u(t,x) &=& \bL_{\mu} u(t,x), \quad  t >0,\\
u(0,x) &=& f(x).
\end{eqnarray*}
Furthermore, the solution has the following spectral representation, for all $t>0$,
\begin{equation*}
P^{\phi_{\alpha}}_t f = \sum_{n=0}^{\infty}\Ep(-(n(n-1)+\lam n)t^{\alpha}) \langle f, \Poly_n^{\lam, \mu} \rangle_{\beta_{\mu}}\Poly_n^{\lam, \mu} \quad \text{in } L^2(\beta_{\mu}).
\end{equation*}

\subsection{Some non-self-adjoint and non-local examples}\label{sec:non-diffusions}

\subsubsection{A generalized Laguerre semigroup}
We next follow ~\cite[Section 3.2]{patie2015spectral} to present a special instance of the so-called generalized Laguerre semigroups. In particular, let $\m \geq 1$ and $P=(P_t)_{t \geq 0}$ be the non-self-adjoint semigroup whose infinitesimal generator is given, for a smooth function $f$, by
\begin{equation*}
\bL_{\m}f(x)=xf''(x)+ \left( \frac{{\m}^2-1}{\m} +1-x \right)f'(x)+\int_0^{\infty}(f(e^{-y}x)-f(x)+yxf'(x))\frac{\m e^{-\m y}}{x}dy, \:x>0.
\end{equation*}
The semigroup $P$ is ergodic with a unique invariant measure, which in this case is an absolutely continuous probability measure with a density denoted by $\nu$ and which takes the form
$$\nu(y)=\frac{(1+y)}{\m+1}\frac{y^{\m-1}e^{-y}}{\Gamma(\m)}, \quad y>0.$$
Moreover, $P_t$ admits the following spectral representation for any $f \in L^2(\nu)$ and $t>0$,
\begin{equation*}
P_t f = \sum_{n=0}^{\infty}e^{-nt} \langle f,  \Nu_n\rangle_{\nu}\Poly_n \quad \text{in } L^2(\nu).
\end{equation*}
Here, $(\Poly_n, \Nu_n)_{n \geq 0}$ form an orthogonal sequence in $L^2(\nu)$, and are expressed in terms of the Laguerre polynomials $\left( \mathcal{L}_n^{(\m)} \right)_{n \geq 0}$ as follows, for  $n\in \N$,
\begin{eqnarray*}
\Poly_n(x) &=& \sum_{k=0}^n (-1)^k {n\choose k}\frac{\Gamma(\m+2)}{\Gamma(\m+k+2)}\frac{\m+k}{\m}x^k = \mathfrak{c}_n(\m+1)\mathcal{L}_{n}^{(\m+1)}(x)-\frac{c_n(\m+1)}{\m}x\mathcal{L}_{n-1}^{(\m+2)}(x) , \label{eq:eigen_perturb}\\
\Nu_n(x)&=&\frac{1}{x+1}\mathcal{L}_{n}^{(\m-1)}(x)+\frac{x}{x+1}\mathcal{L}_{n}^{(\m)}(x).\label{eq:coeigen_perturb}
\end{eqnarray*}
Here, $\mathfrak{c}_n(\m+1)=\frac{\Gamma(n+1)\Gamma(\m+2)}{\Gamma(n+\m+2)}$ and we recall that $\mathcal{L}_n^{(\m)}$ is the Laguerre polynomial of order $\m$,
\begin{equation*}
\calL_n^{(\m)}(x)=\sum_{k=0}^n (-1)^k \binom{n+\m}{n-k}\frac{x^k}{k!}, \quad x>0.
\end{equation*}
Therefore, since Assumption~\ref{assump-P} is satisfied with $\sigma(\bL)=\{\lambda_n=n$, $n \geq 0\}$, Theorem~\ref{thm:Cauchy} implies that $f \in \calD(\bL_{\m})$, $P^{\phi_{\alpha}}_t f$ solves the self-similar Cauchy problem,
\begin{eqnarray*}
\opbPhi u(t,x) &=&\bL_{\m} u(t,x), \quad  t >0,\\
u(0,x) &=& f(x).
\end{eqnarray*}
Furthermore, the solution has the following spectral representation, for all $t>0$,
\begin{equation*}
P^{\phi_{\alpha}}_t f = \sum_{n=0}^{\infty} \Ep(-nt^{\alpha}) \langle f, \Poly_n \rangle_{\nu}\Nu_n \quad \text{in } L^2(\nu).
\end{equation*}

\subsubsection{Generalized Jacobi semigroups}
In this section, following Patie et~al.~\cite{patie2018jacobi}, we provide a short description of a special instance of generalized Jacobi semigroups. In particular, let $\lam > \m> 2$ with $\lam - \m \notin \N$, and $P=(P_t)_{t \geq 0}$ be the non-self-adjoint semigroup associated with the infinitesimal generator given for a smooth function $f$
\begin{eqnarray*}
\bL_{\m} f(x) &=& x(1-x)f''(x) - (\lam x - \m-1)f'(x)- x^{-(\m+1)}\int_0^1 f'(r)r^{\m}dr, \quad x\in E.
\end{eqnarray*}
Then, we have by ~\cite[Proposition 4.1]{patie2018jacobi} that the density of the unique invariant measure of the Markov semigroup $P$ is given by
\begin{equation*}
\beta(y) = \frac{((\lam - \m - 2)y+1)}{(\m+1)(1-y)} \beta_{\m}(y),  \quad y \in (0,1),
\end{equation*}
where $\beta_{\m}$ is the distribution of the beta random variable of  parameters $\m>0$ and $\lam-\m>0$, see~\eqref{eq:beta-dist}. Furthermore, for any $t>0$ and $f\in L^2(\beta)$, $P_t$ admits the following spectral representation
\begin{equation*}
P_t f = \sum_{n=0}^{\infty} e^{-\lambda_n t} \langle f, \Poly_n \rangle_{\beta} \Nu_n \quad \text{in } L^2(\beta)
\end{equation*}
where we recall that $(\lambda_n)_{n \geq 0}$ are defined by~\eqref{Q_eigenv}, and $(\Poly_n,\Nu_n)_{n \geq 0}$ form a biorthogonal sequence in $L^2(\beta)$ and are defined as follows. We have that $\Poly_0 \equiv 1$ and, for $n \geq 1$,
\begin{equation*}
\Poly_n(x)= \frac{n!}{(\m+2)_n} \sqrt{C_n(1)} \left( \frac{\Poly_n^{(\lam,\m+2 )} (x)}{\sqrt{C_n(\m+2)}}  + \frac{x}{\m} \frac{\Poly_{n-1}^{(\lam+1,\m+3)}(x)}{\sqrt{\widetilde{C}_{n-1}(\m+3)}} \right), \ x \in E.
\end{equation*}
making explicit the dependence on the two parameters for the classical Jacobi polynomials~\eqref{J_pol}, and where $\widetilde{C}_{n}(\m+3) = n!(2n+\lam)(\lam+1)_n / (\m+3)_n(\lam-\m-2)_n$ and $C_n$-s are defined by~\eqref{eq:Cn-lam-mo}. For any $n \in \N$ the function $\Nu_n$ is given by
\begin{equation*}
\Nu_n(x) =\frac{1}{\beta(x)} C_{\lam, \m, n} \frac{\sin(\pi(\m-\lam))}{\pi} \sum_{k = 0}^{\infty} \frac{(\m+1)_{k+n}}{(\m+1)_{k\phantom{+n}}} \frac{\Gamma(k+\m-n-\lam+1)}{k!}(k-1)  x^{k+\m}, \ x \in E^o,
\end{equation*}
where $C_{\lam, \m, n} = \m(\lam - 1) \Gamma(\lam+n-1) \sqrt{C_n(1)} (-2)^n / (n!\Gamma(\m+2))$. Hence, since Assumption~\ref{assump-P} is satisfied, Theorem~\ref{thm:Cauchy} implies that for $f \in \calD(\bL_{\m})$, $P^{\phi_{\alpha}}_t f$ solves the self-similar Cauchy problem,
\begin{eqnarray*}
\opbPhi u(t,x) &=&\bL_{\m} u(t,x), \quad  t >0,\\
u(0,x) &=& f(x).
\end{eqnarray*}
Lastly, the solution has the following spectral representation, for all $t>0$,
\begin{equation*}
P^{\phi_{\alpha}}_t f = \sum_{n=0}^{\infty} \Ep(-(n(n-1)+\lam n)t^{\alpha}) \langle f,  \Nu_n\rangle_{\beta}\Poly_n \quad \text{in } L^2(\beta).
\end{equation*}

\bibliographystyle{plain}

\end{document}